\documentclass[12pt]{amsart}
\usepackage{hyperref, tikz}
\usepackage{placeins}

\usepackage{amsthm, amssymb}
\usepackage{caption}
\newtheorem{theorem}{Theorem}[section]

\newtheorem{lemma}[theorem]{Lemma}

\newtheorem{defn}[theorem]{Definition}

\numberwithin{equation}{section}
\title[simultaneous similarity classes of commuting matrices]{Simultaneous Similarity Classes of \\Commuting matrices over a finite field}
\author{Uday Bhaskar Sharma}
\address{The Institute of Mathematical Sciences, Chennai.}
\email{udaybs@imsc.res.in}
\date{\today}
\newcommand{\Cent}{\mathrm{Central}}
\newcommand{\Reg}{\mathrm{Regular}}

\newcommand{\Irr}{\mathrm{Irr}}
\newcommand{\Supp}{\mathrm{Supp}}
\newcommand{\Cal}[1]{\mathcal{#1}}

\newcommand{\F}{\mathbb{F}}
\newcommand{\Z}{\mathbb{Z}}

\newcommand{\lbd}{\lambda}

\newcommand{\rgtr}{\rightarrow}

\newcommand{\sbteq}{\subseteq}

\newcommand{\rmO}{\mathrm{O}}
\begin{document}

\begin{abstract}
This paper concerns the enumeration of isomorphism classes of modules of a polynomial algebra in several variables over a finite field. This is the same as the classification of commuting tuples of matrices over a finite field up to simultaneous similarity. Let $c_{n,k}(q)$ denote the number of isomorphism classes of $n$-dimensional $\F_q[x_1,\dotsc,x_k]$-modules. The generating function $\sum_k c_{n,k}(q)t^k$ is a rational function. We compute this function for $n\leq 4$. We find that its coefficients are polynomial functions in $q$ with non-negative integer coefficients.
\end{abstract}
\maketitle

\section{Introduction}
\subsection{Background} Let $\F_q$ be a finite field of order $q$. Let $\Cal{A}$ be a finite dimensional algebra (with unity) over $\F_q$. Let $\Cal{A}^*$ denote the group of units in $\Cal{A}$. Let $k \geq 1$ be a positive integer, and $\Cal{A}^{(k)}$ be the set of $k$-tuples of elements of $\Cal{A}$, whose entries commute with each other, i.e., $$\Cal{A}^{(k)} = \{(a_1,\ldots, a_k) \in \Cal{A}^k \text{ $\mid$ $a_ia_j = a_ja_i$ for $i \neq j$}\}.$$
Denote by $c_{\Cal{A},k}$, the number of orbits in $\Cal{A}^{(k)}$ for the action of simultaneous conjugation of $\Cal{A}^*$ on it, which is defined below:\\
\begin{defn}\label{SSClass}
For $(a_1, \ldots, a_k) \in \Cal{A}^{(k)}$ and $g \in \Cal{A}^*$: $$g.(a_1,\ldots,a_k) = (ga_1g^{-1}, \ldots, ga_kg^{-1}).$$ This action is called \textbf{simultaneous conjugation}, and the orbits for this action are called \textbf{simultaneous similarity classes}.
\end{defn}

For an element, $a \in \Cal{A}$, let $Z_\Cal{A}(a)$ denote the centralizer algebra of $a$, and $Z_{\Cal{A}^*}(a)$ denote the group of units of $Z_{\Cal{A}}(a)$.
Let $h_{\Cal{A}}(t)$ be the generating function of $c_{\Cal{A},k}$ in $k$: $$h_{\Cal{A}}(t) = 1 + \sum_{k=1}^\infty c_{\Cal{A},k} t^k.$$
Then we have the theorem:
\begin{theorem}\label{ratfun}
For any finite dimensional algebra $\Cal{A}$ over $\F_q$, $h_{\Cal{A}}(t)$ is a rational function.
\end{theorem}
\begin{proof}
 Given $a_1 \in \Cal{A}$, consider the $k$-tuple, $(a_1,a_2,\ldots,a_k)\in \Cal{A}^{(k)}$. That means, $(a_2, \ldots, a_k)\in Z_{\Cal{A}}(a_1)^{(k-1)}$. Thus, the map:
\begin{displaymath}
(a_1,a_2,\ldots,a_k) \mapsto (a_2,\ldots,a_k),
\end{displaymath}
induces a bijection from the set of $\Cal{A}^*$-orbits in $\Cal{A}^{(k)}$, which contain an element whose 1st coordinate is $a_1$, onto the set of orbits in $Z_\Cal{A}(a_1)^{(k-1)}$ for the action of simultaneous conjugation by $Z_{\Cal{A}^*}(a_1)$ on it. Thus, we get:
\begin{displaymath}
c_{\Cal{A},k} = \sum_{Z \sbteq \Cal{A}}s_Zc_{Z, k-1},
\end{displaymath}
where $Z$ runs over subalgebras of $\Cal{A}$, $s_Z$ is the number of similarity classes in $\Cal{A}$ whose centralizer algebra is isomorphic to $Z$, and $c_{Z, k-1}$ is the number of orbits under the action of $Z^*$ (the group of units of $Z$) on $Z^{(k-1)}$ by simultaneous conjugation. Hence, we have
\begin{equation*}
\begin{aligned}
h_{\Cal{A}}(t) &= 1 + \sum_{k=1}^\infty c_{\Cal{A},k} t^k\\
&= 1 + \sum_{k=1}^\infty \left(\sum_{Z \sbteq \Cal{A}}s_Zc_{Z, k-1}\right)t^k \\
&= 1 + \sum_{Z \sbteq \Cal{A}}s_Zt \left(\sum_{k=1}^\infty c_{Z, k-1}t^{k-1}\right) \text{~ ($c_{Z,0} = 1$ for $Z \sbteq \Cal{A}$)}\\
&= 1 + \sum_{Z \sbteq \Cal{A}}s_Zt \left(1 + \sum_{k=1}^\infty c_{Z, k}t^k\right)\\
&= 1 + s_{\Cal{A}}t\left(1 + \sum_{k=1}^\infty c_{\Cal{A},k}t^k\right) + \sum_{Z \subsetneq \Cal{A}}s_Zt \left(1 + \sum_{k=1}^\infty c_{Z,k}t^k\right)\\
&= 1 + s_{\Cal{A}}t.h_{\Cal{A}}(t) + \sum_{Z \subsetneq \Cal{A}}s_Zt. h_Z(t)\\
~~&~\text{( here, $h_Z(t) = 1 + \sum_{k =1}^\infty c_{Z, k}t^k$)}.
\end{aligned}
\end{equation*}
Therefore,
\begin{equation}\label{EGenFun}
(1-s_{\Cal{A}}t)h_{\Cal{A}}(t) = 1  +  \sum_{Z \subsetneq \Cal{A}}s_Zt.h_Z(t).
\end{equation}
The above identity establishes rationality when $\Cal{A}$ is a commutative algebra. When $\Cal{A}$ is commutative, $\Cal{A}^{(k)} = \Cal{A}^k$. As $\Cal{A}$ is commutative, $Z_{\Cal{A}}(a) = \Cal{A}$ for all $a \in \Cal{A}$. Each element of $\Cal{A}$ is a similarity class in $\Cal{A}$. Thus, $s_\Cal{A} = |\Cal{A}|$, and $s_Z = 0$ for $Z \subsetneq \Cal{A}$. We have, $(1 - |\Cal{A}|t)h_\Cal{A}(t) = 1$, $$\text{hence } h_\Cal{A}(t) = \frac{1}{1 - |\Cal{A}|t},$$ which is a rational function.\\

If $\Cal{A}$ is not commutative, then from identity~(\ref{EGenFun}), we are reduced to the case of algebras whose dimension is strictly less than that of $\Cal{A}$. The rationality of $h_\Cal{A}(t)$ follows by induction on the dimension of $\Cal{A}$. When $\Cal{A}$ is 1 dimensional, $\Cal{A} = \F_q$, which is commutative. When $\dim(\Cal{A}) > 1$, assuming induction for algebras with dimension $ < \dim(\Cal{A})$, we get that, for $Z\subsetneq \Cal{A}$, $h_Z(t)$ is rational. As $\Cal{A}$ is finite, it has only a finite number of subalgebras. Hence, from identity~(\ref{EGenFun}), $h_{\Cal{A}}(t)$, being a sum of a finite number of rational functions, is rational.
\end{proof}


\subsection{Matrices Over Finite Fields}
Let $n, k$ be positive integers. Consider $M_n(\F_q)$, the algebra of $n \times n$ matrices over $\F_q$, and $M_n(\F_q)^{(k)}$, the set of $k$-tuples of commuting $n\times n$ matrices over $\F_q$. We have $GL_n(\F_q)$ acting on $M_n(\F_q)^{(k)}$ by simultaneous conjugation (as in Definition~\ref{SSClass}). Let $c_{n,k}(q)$ denote the number of simultaneous similarity classes in $M_n(\F_q)^{(k)}$.\\

Let $h_n(t)$ denote the generating function of $c_{n,k}(q)$ in $k$. $$h_n(t) = 1 + \sum_{k = 1}^\infty c_{n,k}(q)t^k.$$
As $M_n(\F_q)$ is a finite dimensional algebra over $\F_q$, we get from Theorem~\ref{SSClass}:
\begin{theorem}\label{gen_fun}
For any positive integer $n$, $h_n(t)$ is a rational function.
\end{theorem}

Next, we define the following:
\begin{defn}\label{dtypes0} We say that two simultaneous similarity classes of tuples of commuting matrices are of the same \emph{\textbf{similarity class type}} (or just \emph{\textbf{type}}),  if their centralizers are isomorphic. \end{defn}
We will discuss these \emph{\textbf{similarity class types}} in detail in Section~\ref{S2}.\\

While calculating $c_4(k,q)$ for $k \geq 1$ (Section~\ref{S44}), we come across some new types of simultaneous similarity classes of pairs of commuting matrices of  $M_4(\F_q)$, i.e., the centralizers of pairs of these types are not isomorphic to the centralizers of any of the similarity classes in $M_4(\F_q)$. These new types are dealt with in the subsection~\ref{S44new} of Section~\ref{S44}.\\

In this paper, we compute $h_n(t)$ for $n = 2, 3, 4,$ and the results are given in Table \ref{Tgen}
\begin{table}[h!]
\def\arraystretch{1.5}
 \begin{equation*}\begin{array}{cc} \hline
n & h_n(t)\\ \hline
1 & \frac{1}{1-qt} \\[0.3em]
2 & \frac{1}{(1-qt)(1-q^2t)}\\[0.3em]
3 & \frac{1+q^2t^2}{(1-qt)(1-q^2t)(1-q^3t)}\\[0.6em]
4 & \left(\frac{1+q^2t+2q^2t^2+q^3t^2+2q^4t^2+q^6t^3}{(1-qt)(1-q^2t)(1-q^3t)(1-q^4t)(1-q^5t)} \right)\\
& - \left(\frac{q^5t+q^7t^2+q^3t^3+2q^7t^3+2q^9t^3+q^{10}t^4}{(1-qt)(1-q^2t)(1-q^3t)(1-q^4t)(1-q^5t)} \right)\\[0.3em] \hline
\end{array}\end{equation*}
\caption{Generating functions for $c_{n,k}$ for $n= 1,2,3,4$}
\label{Tgen}
\end{table}

Our calculations are used to prove the following result: \begin{theorem}\label{Tmain}
 For each $n$ in $\{2,3,4\}$ and $k \geq 1$, there exists a polynomial, $Q_{n,k}(t) \in \Z[t]$, with non-negative integer coefficients such that $c_{n,k}(q) = Q_{n,k}(q)$, for every prime power $q$.
\end{theorem}

Let $R$ be a discrete valuation ring with maximal ideal $P$ and residue field $R/P \cong \F_q$. The results of Singla \cite{Singla}, Jambor and Plesken \cite{Jambor} show that $c_{n,2}(q)$ is the number of similarity classes of matrices in $M_n(R/P^2)$. Comparing the results in this paper with those of Avni, Onn, Prasad and Vaserstein \cite{AOPV}, and Prasad, Singla and Spallone \cite{PSS}, we find that the number of similarity classes in $M_3(R/P^k)$ is equal to $c_{3,k}(q)$ for all $k$. The calculations of this paper and the results of the papers cited above lead us to conjecture the following: \begin{itemize}
\item For all positive integers $n, k$, there exists a polynomial $Q_{n,k}(t)$ with non-negative integer coefficients such that $c_{n,k}(q) = Q_{n,k}(q)$.
\item $c_{n,k}(q)$ is the number of conjugacy classes in $M_n(R/P^k)$.
\end{itemize}

\subsection*{Notations} We will be using these notations throughout this paper.
$\F_q$ denotes a finite field of order $q$. Let $n$ and $k$ be positive integers. Then, for any matrix $A \in M_n(\F_q)$, let $Z(A)$ denote the centralizer algebra of $A$, and $Z(A)^*$ denote the centralizer group of $A$. For any tuple $(A_1,\ldots,A_k) \in M_n(\F_q)^{(k)}$, the common centralizer, $\cap_{i=1}^kZ(A_i)$, of $A_1,\ldots,A_k$, is denoted by $Z(A_1,\ldots, A_k)$, and $Z(A_1,\ldots,A_k)^*$ denotes the group of units of $Z(A_1,\ldots, A_k)$.
\section{Similarity Class Types}\label{S2}
Given $A \in M_n(\F_q)$ and $x \in \F_q^n$, define for any polynomial $f(t) \in \F_q[t]$, $f(t).x = f(A)x$. This endows $\F_q^n$ with an $\F_q[t]$-module structure, denoted by $M^A$. It is easy to check that for matrices $A$ and $B$, $$M^A \cong M^B \Leftrightarrow A = gBg^{-1}\text{ for some $g\in GL_n(\F_q)$}$$ We can easily see that, $End_{\F_q[t]}(M^A) = Z(A)$, the centralizer of $A$ in $M_n(\F_q)$.\\

If $A$ is a block diagonal matrix $\begin{pmatrix}B &0\\0&C\end{pmatrix}$, where $B$ and $C$ are square matrices whose characteristic polynomials are coprime, we write $A$ as $B\oplus C$, and $$M^{B\oplus C} = M^A \cong M^{B} \oplus M^{C},$$ and it can be easily shown that $Z(A)$ is isomorphic to $Z(B) \oplus Z(C)$.\\

Next, we have the Jordan decomposition of $M^A$ for which we need:
\begin{defn}
Let $p$ be an irreducible polynomial in $\F_q[t]$, then the submodule $$M^{A_p} = \left\{x \in M^A \text{ : $p(t)^r.x =0$ for some $r \geq 1$}\right\}$$ is called the \textbf{$p$-primary} part of $M^A$.
\end{defn}
Let $\Irr(\F_q[t])$ denote the set of irreducibles in $\F_q[t]$. Then by the primary decomposition theorem, $M^A$ has the decomposition, $$M^A = \bigoplus_{p \in \Irr(\F_q[t])}M^{A_p},$$ which is over a finite number of irreducibles since $M^A$ is finitely generated. Then by Structure~Theorem (see Dummit and Foote \cite{Dummit2004}) of finitely generated modules over a PID, for each $p$, $M^{A_p}$ has the decomposition, $$\frac{\F_q[t]}{p^{\lbd_1}} \oplus \frac{\F_q[t]}{p^{\lbd_2}} \oplus \cdots, $$ where $\lbd_1 \geq \lbd_2 \geq \cdots$ are positive integers. Let $\lbd = (\lbd_1, \lbd_2, \ldots)$. $\lbd$ is a partition. The primary decomposition together with the Structure Theorem decomposition of each primary part gives the decomposition: $$\bigoplus_{p \in \Irr(\F_q[t])}\left(\frac{\F_q[t]}{p^{\lbd_1}} \oplus \frac{\F_q[t]}{p^{\lbd_1}} \oplus \cdots \right).$$ This decomposition is the \textbf{Jordan decomposition}.\\

 This gives a bijection between similarity classes in $M_n(\F_q)$ and the set of maps, $\nu$, from $\Irr(\F_q[t])$ to the set of partitions, $\varLambda$.\\

 Now, for any $\nu:\Irr(\F_q[t])\rgtr \varLambda$, let $\Supp(\nu)$ denote the set of irreducible polynomials $p(t)$ for which $\nu(p)$ is a non-empty partition. Clearly $\Supp(\nu)$ is a finite set. For each partition, $\mu$, and each $d \geq 1$, let $r_{\nu}(\mu,d)$ be: $$r_{\nu}(\mu,d) = |\{p(t) \in \Irr(\F_q[t])\text{ : $\deg(p) = d$ and  $\nu(p) = \mu$} \}|$$ This puts us in a position to define \textit{Similarity Class Types}.
 \begin{defn}\label{dtypes1}
 Let $A$ and $B$ be two similarity classes in $M_n(\F_q)$, and let $\nu^{(A)}$ and $\nu^{(B)}$ be the maps from $\Irr(\F_q[t])  \rgtr \varLambda$ corresponding to $A$ and $B$ respectively. We say that $A$ and $B$ are of the same \textbf{Similarity Class Type} if for each partition, $\lbd$, and each $d \geq 1$, $r_{\nu^{(A)}}(d,\lbd)$ is equal to $r_{\nu^{(B)}}(d,\lbd)$ (See Green \cite{Green}).                                                                                                                                                                                                                                                                                                                                                                                                                                                                                                                                                                                                                                                                                                                                                                                                                                                                                                                        \end{defn}
We shall denote a similarity class type by $${\lbd^{(1)}}_{d_1},\ldots,{\lbd^{(l)}}_{d_l},$$ where ${\lbd^{(1)}},\ldots, {\lbd^{(l)}}$ are partitions and $d_i \geq 1$ for $1 \leq i \leq l$, such that $$\sum_{i=1}^l|{\lbd^{(i)}}|d_i = n.$$
For example, in $M_2(\F_q)$, there are four similarity class types, which are described in the table below:
\def\arraystretch{1.3}
\begin{equation*}
\begin{array}{|c| c|}\hline
\mathrm{Type} & \mathrm{Description~of~the~type}\\ \hline
(1,1)_1 & \text{ $\lbd^{(1)} = (1,1)$, $d_1 =1$} \\[0.3em] \hline
(2)_1 & \text{ $\lbd^{(1)} = (2)$, $d_1 = 1$}\\[0.3 em]\hline
(1)_1(1)_1 & \text{ $\lbd^{(1)} = (1), d_1 = 1$}\\
&\text{ $\lbd^{(2)} = (1)$, $d_2 = 1$}\\[0.5em]\hline
(1)_2 & \text{ $\lbd^{(1)} = (1)$, $d_1 = 2$}\\ \hline
\end{array}
\end{equation*}
So, for a similarity class, $\nu: \Irr(\F_q[t]) \rgtr \varLambda$, such that $\Supp(\nu) = \{f_1,\ldots,f_l\}$, where $deg(f_i) = d_i$ and $\nu(f_i) = {\lbd^{(i)}}$, the similarity class type is $${\lbd^{(1)}}_{d_1},\ldots,{\lbd^{(l)}}_{d_l}.$$

\begin{defn}
\begin{enumerate}
\item We say that a matrix $A$ is of the $\mathrm{Central}$ type if it is of the similarity class type $$(\underbrace{1,\ldots,1}_{n-\text{ones}})_1$$
\item And of the $\mathrm{Regular/Cyclic}$ type if it is of the class type $${\lbd^{(1)}}_{d_1},\ldots, {\lbd^{(l)}}_{d_l}$$ where for each $i = 1,\ldots,l$, the partition ${\lbd^{(i)}}$ has only one part. For all such types, $\F_q^n$ has a cyclic vector.
\end{enumerate}
\end{defn}
Before going to the next section, we shall define types for commuting tuples of matrices:
\begin{defn}\label{Type}
Let $(A_1,\ldots,A_k)$ be a $k$-tuple and $(B_1,\ldots,B_l)$, an $l$-tuple of commuting matrices. We say that they are of the same \textbf{similarity class type} if their respective common centralizers $Z(A_1,\ldots,A_k)$ and $Z(B_1,\ldots,B_l)$ are isomorphic in $M_n(\F_q)$.                                                                                                               \end{defn}

The above definition of types for tuples is a more precise version of Definition~\ref{dtypes0}, and is consistent with the Definition~\ref{dtypes1} because, matrices $A$ and $B$ are of the same type if and only if their centralizers, $Z(A)$ and $Z(B)$, are isomorphic (see the definition of orbit-equivalent by Ravi S. Kulkarni in \cite{RSK} or the definition of $z$-equivalent by Rony Gouraige \cite{Gouraige}). If the centralizer, $Z(A_1,\ldots,A_k)$, of a $k$-tuple, $(A_1,\ldots,A_k)$, for $k \geq 2$, is isomorphic to that of a matrix, $A \in M_n(\F_q)$ (of some type $\tau$), we say that the simultaneous similarity class of $(A_1,\ldots,A_k)$ is of type $\tau$. So if the centralizer, $Z(B_1,\ldots,B_l)$, of $(B_1,\ldots,B_l)$ is isomorphic to $Z(A_1,\ldots,A_k)$, then it is isomorphic to the centralizer of $A$; hence $(B_1,\ldots,B_l)$ too is of type $\tau$. If $Z(A_1,\ldots,A_k)$ is not isomporphic to the centralizer of any matrix in $M_n(\F_q)$, we have a new type of similarity class.\\

\section{The $2\times 2$ Case}\label{S22}
We shall examine the similarity classes of commuting $k$-tuples of $2\times 2$ matrices over $\F_q$ in this section. Before going ahead, we shall define the \textit{branch} of a similarity class type.
\begin{defn}
Given a matrix $A$ of a type $\tau$ in $M_n(\F_q)$, let $Z(A)$ be its centralizer. We saw in the proof of Theorem~\ref{ratfun} that, counting the number of simultaneous similarity classes of pairs with the first coordinate $A$, is the same as counting the similarity classes in $Z(A)$ under the conjugation by its group of units, $Z(A)^{*}$.\\

So, for each $B$ in $Z(A)$, its centralizer subalgebra in $Z(A)$ is the common centralizer, $Z(A,B)$, of $A$ and $B$. Let $\rho$ denote the class type of the similarity class of $(A,B)$ (in the sense of Definition~\ref{Type}). Then we say that the type $\rho$, is a \emph{\textbf{branch}} of $\tau$. \end{defn}

Hence for a matrix, $A$, of type $\tau$, the number of branches of various types will be determined modulo the conjugation action of $Z(A)^*$ on $Z(A)$. We will use this same method in finding the branching rules in Sections~\ref{S33}~and~\ref{S44}. \\

In $M_2(\F_q)$, there are two kinds of similarity classes:
\begin{enumerate}
\item The $\mathrm{Central}$ type which is $(1,1)_1$.
\item The $\mathrm{Regular/Cyclic}$ types, where $\mathbb{F}_q^2$ has a cyclic vector.
\end{enumerate}
\begin{lemma}\label{L2cent}
For a matrix, $A$, of the $\Cent$ type, the branches are given in the table below:\\
\def\arraystretch{1.5}
 \begin{equation*}\begin{array}{lc}\hline
 \mathrm{Type} & \mathrm{Number~of~Branches}\\ \hline
\Cent & q\\
\Reg & q^2\\ \hline
\end{array}
\end{equation*}
\end{lemma}
\begin{proof}
 When $A$ is of the $\Cent$ type, $Z(A) = M_2(\F_q)$ and $Z(A)^* = GL_2(\F_q)$. So we only need to enumerate the similarity classes in $M_2(\F_q)$, which leads to the table shown in the statement of this lemma.
 \end{proof}
\begin{lemma}\label{L2reg}
 A matrix of any of the $\Reg$ types has $q^2$ regular type of branches.
\end{lemma}
\begin{proof}
 The centralizer algebra of a regular type of matrix say $A$ is $$\{a_0I + a_1A\text{ : $a_0, a_1 \in \F_q$} \}$$ which is a commutative algebra and thus each orbit under the conjugation action of $Z(A)^*$ on $Z(A)$ is a singleton . Hence, for any $B \in Z(A)$, $Z(A,B)=Z(A)$. There are $q^2$ such similarity classes; hence $q^2$ $\Reg$ branches.
\end{proof}
So we see no new types of similarity classes here. Arranging the two types in the order: $\{\Cent, \Reg\}$, we shall write down the branching matrix $\Cal{B}_2 = [b_{ij}]$, indexed by the types. For each $i$ and $j$, $b_{ij}$ is the number of type $i$ branches of a tuple of type $j$. So here, the branching matrix is:
$$\Cal{B}_2 = \begin{pmatrix}q&0\\q^2&q^2\end{pmatrix}$$
We have
 $$c_{2,k}(q) = \begin{pmatrix}1&1\end{pmatrix}\Cal{B}_2^k\begin{pmatrix}1&0\end{pmatrix}^T.$$
From the entries of $\Cal{B}_2$ and the above equation, it is clear that $c_{2,k}(q)$ is a polynomial with non-negative integer coefficients. The generating function, $h_2(t)$, of the $c_{2,k}(q)$ is:
$$h_2(t) = 1+ \sum_{k = 1}^{\infty}c_{2,k}(q)t^k = \begin{pmatrix}1&1\end{pmatrix} (I - t\Cal{B}_2)^{-1} \begin{pmatrix}1&0\end{pmatrix}^T,$$ which is equal to $$\frac{1}{(1-qt)(1-q^2t)}.$$
\section{The $3\times3$ Case}\label{S33}
In $M_3(\F_q)$ we have the following types of similarity classes:
\begin{enumerate}
\item The $\Cent$ type, $(1,1,1)_1$.
\item The \textbf{$(2,1)$ nilpotent type}: $(2,1)_1$.
\item The \textbf{$(2,1)$ semi-simple type}: $(1,1)_1(1)_1$.
\item The $\Reg$ types, where $\F_q^3$ has a cyclic vector.
\end{enumerate}
 We now proceed to explain the branching rules.
\begin{lemma}\label{L111}
 For a matrix $A$ of the $\Cent$ type, the branching rules are shown in the table below.
\def\arraystretch{1.24}
 \begin{equation*}\begin{array}{cc}\hline
 \mathrm{Type} & \mathrm{Number~of~Branches}\\ \hline
\Cent & q\\
(2,1)_1 & q \\
(1,1)_1(1)_1 &q^2-q \\
\Reg & q^3\\ \hline
\end{array}
\end{equation*}
\end{lemma}
\begin{proof}
 Since $A$ is of $\Cent$ type, $Z(A)$ is $M_3(\F_q)$. Enumeration of the similarity class types in $M_3(\F_q)$ gives us the table above.
\end{proof}

\begin{lemma}\label{L21}
 For matrix, $A$, of the $(2,1)$-nilpotent type i.e., the type $(2,1)_1$, the branching rules are given in the table below.
\def\arraystretch{1.5}
 \begin{equation*}\begin{array}{cc}\hline
 \mathrm{Type} & \mathrm{Number~of~Branches}\\ \hline
(2,1)_1 & q^2 \\
\Reg & q^3 + q\\ \hline
\end{array}
\end{equation*}
\end{lemma}
\begin{proof}
If $A$ is of type $(2,1)_1$, we shall consider its canonical form, $\begin{pmatrix}a&1&0\\0&a&0\\0&0&a\end{pmatrix}$, $a \in \F_q$. Thus,

$$Z(A) = \left\{\begin{pmatrix}a_0&a_1&b\\ 0&a_0&0\\0&c&d\end{pmatrix}\text{ $a_0,a_1,b,c,d \in \F_q$ }\right\}.$$

Consider $B \in Z(A)$.
$$B = \begin{pmatrix}a_0&a_1&b\\0&a_0&0\\0&c&d\end{pmatrix}.$$

Let $X \in Z(A)^*$ be $$X = \begin{pmatrix}x_0&x_1&y\\0&x_0&0\\0&z&w\end{pmatrix},$$ where $x_0, w \neq 0$. Let $B' = XBX^{-1}$. Then we have

\begin{equation}\label{E21main}
X\begin{pmatrix}a_0&a_1&b\\0&a_0&0\\0&c&d\end{pmatrix} = \begin{pmatrix}a'_0&a'_1&b'\\0&a'_0&0\\0&c'&d'\end{pmatrix}X
\end{equation}
From equation~(\ref{E21main}), we get $x_0a_0 = a'_0x_0$ and $wd = d'w$, so $a_0 = a'_0$ and $d = d'$, and we have the following equations:
\begin{eqnarray}
x_0a_1 + cy &=& a'_1x_0 + b'z \label{E211}\\
x_0b + yd &=& a_0y + b'w\label{E212} \\
a_0z + cw &=& c'x_0 + dz\label{E213}.
\end{eqnarray}
So, we look at the two cases over here: $a_0 = d$ and $a_0 \neq d$.\\

\textbf{When $a_0 = d$:} From equations~(\ref{E212})~and~(\ref{E213}), we get, $x_0b = b'w$ and $cw = c'x_0$. So, we look at two sub cases here: $b=c=0$ and $(b,c) \neq (0,0)$.\\

When $b=c=0$, equation~(\ref{E211}) is reduced to $x_0a_1 = a'_1x_0$, which gives us $a'_1 = a_1$. So $B$ is $$\begin{pmatrix}a_0&a_1&0\\0&a_0&0\\0&0&a_0\end{pmatrix};$$ therefore the centralizer, $Z(A, B)$, of $B$ in $Z(A)$, is $Z(A)$ itself. Therefore the pair $(A,B)$ is of similarity class type $(2,1)_1$, and there are $q^2$ such similarity classes.\\

When $(b,c)\neq (0,0)$: Suppose $b \neq 0$. Then, we can make $b' = 1$ in equation~(\ref{E212}) by choosing a suitable $x_0$. So, letting $b = b' = 1$, we get $x_0 = w$, and equation~(\ref{E213}) gives us $c = c'$. Equation~(\ref{E211}) becomes $x_0a_1 + cy = a'_1x_0 + z$, so we choose $z$ such that $x_0a_1 = 0$;  thus $a_1 =0$. So  $$ B = \begin{pmatrix}a_0&0&1\\0&a_0&0\\0&c&a_0\end{pmatrix}$$ whose centralizer in $Z(A)$ is $$\left\{\begin{pmatrix}x_0&x_1&y\\0&x_0&0\\0&cy&x_0\end{pmatrix} \text{ : $x_0,x_1,y \in \F_q$}\right\},$$ which is similar to the centralizer of a $\Reg$ nilpotent (type $(3)_1$) type of matrix. This is because we can switch the 2nd and 3rd rows (resp. columns) to get a matrix that commutes with a $\Reg$ nilpotent matrix. Hence the branch $(A,B)$ is of $\Reg$ type and there are $q^2$ such branches.\\

 If $b = 0$, then $c \neq 0$. In equation~(\ref{E213}), choose $w = x_0/c$ to get $c' = 1$. Therefore, letting $c = c' = 1$, we get $w = x_0$ and equation~(\ref{E211}) becomes $x_0a_1+y = a'_1x_0$. Now choose $y = a'_1x_0$ so that $x_0a_1 = 0$ and thus $a_1 = 0$. Thus $$B = \begin{pmatrix}a_0&0&0\\0&a_0&0\\0&1&a_0\end{pmatrix},$$ and its centralizer in $Z(A)$ is $$\left\{\begin{pmatrix}x_0&x_1&0\\0&x_0&0\\0&z&x_0\end{pmatrix} \text{ : $x_0,x_1,z \in \F_q$}\right\},$$ which can again be seen as the centralizer of a $\Reg$ nilpotent($(3)_1$) type of matrix. We have $q$ more $\Reg$ branches.\\

\textbf{When $a_0 \neq d$:} In this case, in equation~(\ref{E212}), we can find $y$ such that $b = 0$ and in equation~(\ref{E213}), we can find a $z$ such that $c = 0$. Therefore, equation~(\ref{E211}) is reduced to $x_0a_1 = a'_1x_0$. Thus $a'_1 = a_1$. So $B$ is $$\begin{pmatrix}a_0&a_1&0\\0&a_0&0\\0&0&d\end{pmatrix},$$ and its centralizer in $Z(A)$ is $$\left\{ \begin{pmatrix}x_0&x_1&0\\0&x_0&0\\0&0&w\end{pmatrix} \text{ : $x_0,x_1,w \in \F_q$}\right\},$$ which is that of a matrix of a $\Reg$ type i.e., $(2)_1(1)_1$. We have $q^2(q-1)$ $\Reg$ branches. So we have a total of $(q^3-q^2) + q^2 +q = q^3+q$ $\Reg$ branches.
\end{proof}
\begin{lemma}\label{L1101}
If $A$ is a matrix whose similarity class is of the type $(1,1)_1(1)_1$ i.e., the $(2,1)$-semisimple type, then it has
\begin{itemize}
 \item $q^2$ branches of the $(2,1)$ semisimple type, $(1,1)_1(1)_1$.
\item $q^3$ branches of the $\mathrm{Regular}$ types.
\end{itemize}
 \end{lemma}
\begin{proof}
 A matrix $A$ of similarity class which is of type $(1,1)_1(1)_1$ is of the form $A' \oplus A''$ where $A'$ is a $2\times 2$ matrix of the $\Cent$ type and $A''$ is a $1 \times 1 $ matrix. So the centralizer algebra of $A$ is of the form $Z(A') \oplus Z(A'')$ where $Z(A')= M_2(\F_q)$. Now $A'$ has $q$ branches of the $\Cent$ type, and $q^2$ branches of the $\Reg$ type (see Lemma~\ref{L2cent}). And $A''$ has $q$ branches. The branches of $A$, being in $Z(A') \oplus Z(A'')$, will be of the form $B' \oplus B''$, where $B'$ is a branch of $A'$, and $B''$ is a branch of $A''$. This leaves us with $q\times q= q^2$ branches of the type $(1,1)_1(1)_1$ and $q^2 \times q= q^3$ $\Reg$ branches.
\end{proof}
\begin{lemma}\label{L3reg}
 If $A$ is a matrix of a $\mathrm{Regular}$ type, then it has $q^3$ branches of that same $\mathrm{Regular}$ type.
\end{lemma}
\begin{proof}
 If $A$ is of a $\Reg$ type, its centralizer algebra $Z(A)$ is $$\{ a_0I + a_1A + a_2A^2\text{ : $a_0,a_1,a_2 \in \F_q$}\}.$$ It is a commutative algebra of dimension 3. Thus for any $B \in Z(A)$, $Z(A,B) =Z(A)$ . Therefore $(A,B)$ is of the $\Reg$ type and the number of such branches is $q^3$.
\end{proof}
We shall arrange the types in the order: $$\{ \Cent, (2,1)_1, (1,1)_1(1)_1, \Reg\},$$ and write down the branching matrix, $\Cal{B}_3 = [b_{ij}]$ indexed by the types in that order. Here, an entry, $b_{ij}$, of $\Cal{B}_3$ is the number of type $i$ branches of a type $j$ similarity class. So, $$\Cal{B}_3 = \begin{pmatrix}q&0&0&0\\q&q&0&0\\q^2-q&0&q^2&0\\q^3&q^3+q&q^3&q^3\end{pmatrix}.$$
To make things easier, we shall interpret the branching rules in terms of what we call \emph{rational canonical form (rcf)} types, which we shall briefly discuss now.\\

The similarity class types in $M_n(\F_q)$ can be further classified into these rcf-types. The definition of rcf types is given below:
\begin{defn}
 As $M^A$ is a finitely generated $\F_q[t]$-module, by the Structure Theorem (see Jacobson \cite{Jac}), $M^A$ has the decomposition \begin{equation}\label{ES33}\frac{\F_q[t]}{f_1(t)}\oplus \cdots \oplus \frac{\F_q[t]}{f_r(t)},\end{equation}                                                                                                                                                                                                       where $f_r(t) \mid f_{r-1}(t) \mid \cdots \mid f_1(t)$. Let $l_i$ be the degree of $f_i$. Then $\lbd = (l_1,\ldots,l_r)$ is a partition of $n$ and we say that $A$ is of {\bfseries rational canonical form (rcf)}-type $\lbd$.
\end{defn}
Let $A$ be a matrix with similarity class type, ${\lbd^{(1)}}_{d_1},\ldots,{\lbd^{(l)}}_{d_l}$, where for each $i$, ${\lbd^{(i)}} =(\lbd^{(i)}_1,\lbd^{(i)}_2,\ldots)$. Then there are irreducible polynomials $p_1(t),\ldots,p_l(t)$ with degrees $d_1,\ldots,d_l$ respectively such that $$M^A = \bigoplus_{i=1}^l\left(\frac{\F_q[t]}{p_i(t)^{\lbd^{(i)}_1}} \oplus \frac{\F_q[t]}{p_i(t)^{\lbd^{(i)}_2}} \oplus \cdots \right)$$
Then, in the structure theorem decomposition of $M^A$ as given in equation~(\ref{ES33}), we have (see \cite{Jac}) $$f_j(t) = p_1(t)^{\lbd^{(1)}_j}p_2(t)^{\lbd^{(2)}_j}\cdots p_l(t)^{\lbd^{(l)}_j}$$ Hence for each $j$, the degree $l_j$ of $f_j$ is $$\sum_{i=1}^l\lbd^{(i)}_jd_i.$$ Hence, $(l_1,l_2,\ldots)$ is $$\nu = \left(\sum_{i=1}^l\lbd^{(i)}_1d_i,\sum_{i=1}^l\lbd^{(i)}_2d_i,\ldots\right).$$ This partition $\nu$ is called the rcf-type of the similarity class type, $${\lbd^{(1)}}_{d_1},\ldots,{\lbd^{(l)}}_{d_l}.$$
 Thus, in the $3\times 3$ case, the rcf types are $(1,1,1)$, $(2,1)$ and $(3)$. We see that
 \begin{enumerate}
 \item The $\Cent$ type $(1,1,1)_1$ is the only class type with rcf type $(1,1,1)$.
 \item Similarity class types: $(2,1)_1$ (i.e., the $(2,1)$-nilpotent type) and $(1,1)_1(1)_1$ (the $(2,1)$-semisimple type) are of the rcf type $(2,1)$.
 \item The $\Reg$ types are of rcf type $(3)$.
 \end{enumerate}                                                                                                                                                                                     We know that there are $q^2$ classes with rcf-type $(2,1)$ in $M_3(\F_q)$, of which $q^2-q$ of them are of the semi-simple type $(1,1)_1(1)_1$ and $q$ of them are of the nil-potent type $(2,1)_1$. Hence a class of rcf type $(2,1)$ is of type $(1,1)_1(1)_1$ with probability $\displaystyle\frac{q-1}{q}$ and it is of type $(2,1)_1$ with probability $\displaystyle\frac{1}{q}$.\\
So, the number of $\Reg$ branches that a matrix of rcf type $(2,1)$ has on an average is $$\frac{q-1}{q}\times q^3 + \frac{1}{q}\times(q^3+q)$$ which is equal to $q^3 +1$. The average number of rcf type $(2,1)$ branches of the rcf type $(2,1)$ is $$\frac{q-1}{q}\times q^2 + \frac{1}{q}\times q^2$$ which is equal to $q^2$. So, our branching matrix is reduced to
$$\Cal{B}_3 = \begin{pmatrix}q&0&0\\q^2&q^2&0\\q^3&q^3+1&q^3\end{pmatrix}$$

In general, for a given rcf $\lbd$, let $p_{\tau}^\lbd$ be the probability of a class of rcf type $\lbd$, being of similarity class type $\tau$. Then, for rcf types, $\mu$ and $\lbd$, the average number of rcf-type $\mu$ branches of an rcf-type $\lbd$ similarity class is $$b_{\mu\lbd} = \sum_{rcf(\tau) = \lbd}p_{\tau}^\lbd\left(\sum_{rcf(\gamma) = \mu}b_{\gamma\tau}\right).$$ Now that we have reduced $\Cal{B}_3$, we have the theorem:
\begin{theorem}\label{S33T}
 The number of similarity classes $c_{3,k}(q)$ of commuting $k$-tuples over $\F_q$ for $k\geq 2$ is given by $$\begin{pmatrix}1&1&1\end{pmatrix}\Cal{B}_3^k\begin{pmatrix}1&0&0\end{pmatrix}^{T}$$
\end{theorem}
Table \ref{Tab33} shows $c_{3,k}(q)$ calculated for $k = 1,2,3$.
\begin{table}[h!]
\def \arraystretch{1.5}
\begin{equation*}
\begin{array}{lc}\hline
k & c_{3,k}(q)\\ \hline
1& q^{3} + q^{2} + q \\
2&q^{6} + q^{5} + 2 q^{4} + q^{3} + 2 q^{2}\\
3 &q^{9} + q^{8} + 2 q^{7} + 2 q^{6} + 3 q^{5} + 2 q^{4} + 2 q^{3}\\ \hline
\end{array}\end{equation*}
\caption{$c_{3,k}(q)$ for $k = 1,2,3,4$}\label{Tab33}
\end{table}
\\
As each entry of $\Cal{B}_3$ is a polynomial in $q$, with non-negative integer coefficients, we get that $c_{3,k}(q)$ is a polynomial in $q$ with non-negative integer coefficients. This proves Theorem~\ref{Tmain} for this case. The generating function $h_3(q)$, of $c_{3,k}(q)$ is:
$$1 + \sum_{k=1}^\infty c_{3,k}(q)t^k = \begin{pmatrix}1&1&1\end{pmatrix} ( I - \Cal{B}_3)^{-1}\begin{pmatrix}1&0&0\end{pmatrix}^{T},$$
which is: $$\frac{1 + q^2t^2}{(1-qt)(1-q^2t)(1-q^3t)}.$$
\section{The $4\times 4$ Case}\label{S44}
In the $4\times 4$ case, we have 22 similarity class types, whose branching rules we need to check. Table~\ref{Tabl4} shows the rcf types and the similarity class types of each rcf-type listed below it.\\
\noindent
\begin{table}[h!]
\def \arraystretch{1.1}
\begin{equation*}
\begin{array}{ | c | c | c| c | c |} \hline
     (1,1,1,1)& (2,1,1) & (2,2) & (3,1) & (4)\\ \hline
     (1,1,1,1)_1& (2,1,1)_1 & (2,2)_1 & (3,1)_1 &\text{$\Reg$~types,}\\
       ~           & (1,1,1)_1(1)_1 & (1,1)_1(1,1)_1 & (2,1)_1(1)_1& \text{where $\F_q^4$}\\
     & & (1,1)_2 & (2)_1(1,1)_1 & \text{has a cyclic}\\
     & & & (1,1)_1(1)_1(1)_1&\text{vector.}\\
     & & & (1)_2(1,1)_1&\\
\hline
\end{array}
\end{equation*}
\caption{rcf's and similarity class types of $4\times4$ matrices}\label{Tabl4}
\end{table}

Before we move ahead, we shall give a broader definition of $\Reg$ type.
\begin{defn} We say that a $k$-tuple of commuting matrices is of $\Reg$ type if its common centralizer algebra is a commutative algebra of dimension 4 or conjugate to the centralizer of a $\Reg$ type from $M_4(\F_q)$ (centralizers of $\Reg$ types in $M_n(\F_q)$ are 4-dimensional and commutative).\end{defn}

We shall first state the branching rules of the $\Reg$ and the $\Cent$ types and discuss the branching rules of the other types in different subsections of this section.
\begin{lemma}\label{LReg4}
 If $A$ is a matrix of a $\mathrm{Regular}$ type, then it has $q^4$ branches of that same $\Reg$ type.
\end{lemma}
\begin{proof}
 The centralizer $Z(A)$ of $A$, is the algebra of polynomials in $A$ and it is a commutative algebra. Since the characteristic polynomial of $A$ is of degree 4, the algebra $Z(A)$ is 4-dimensional. So, for each $B \in Z(A)$, $(A,B)$ is a branch of the $\Reg$ type. Therefore we have $q^4$ $\Reg$ branches.
\end{proof}
\begin{lemma}\label{L1111}
For $A$ of the $\mathrm{Central}$ type, its branches are given in the table below:
\def\arraystretch{1.3}
 \begin{equation*}\begin{array}{|cc|cc|}\hline
 \mathrm{Type} & \mathrm{No.~of~Branches}& \mathrm{Type} & \mathrm{No.~of~Branches}\\ \hline
\Cent & q &(3,1)_1 & q\\
(2,1,1)_1 & q&(2,1)_1(1)_1 & q^2-q\\
(1,1,1)_1(1)_1 & q^2-q &(1,1)_1(1)_1(1)_1 &\frac{q(q-1)(q-2)}{2}\\
(2,2)_1 & q &(1,1)_1(2)_1 & q^2-q\\
(1,1)_1,(1,1)_1 &\frac{q^2-q}{2} &(1,1)_1(1)_2 & \frac{q^3-q^2}{2}\\
(1,1)_2 & \frac{q^2-q}{2} & \Reg & q^4\\ \hline
\end{array}
\end{equation*}
\end{lemma}

\begin{proof}
 As $A$ is of $\mathrm{Central}$ type, its centralizer algebra $Z(A)$ is the whole of $M_4(\F_q)$ and the centralizer group is the whole of $GL_4(\F_q)$. Enumerating the similarity classes of $M_4(\F_q)$ gives the above table.
\end{proof}

\subsection{Branching Rules of the non-primary, non-$\Reg$ types.}
Any non-primary similarity class type of $M_n(\F_q)$ is of the form $${\lbd^{(1)}}_{d_1}\cdots {\lbd^{(l)}}_{d_l}$$ where $l \geq 2$. Hence the centralizer algebra of matrices of such types consist of block matrices of the form $$\begin{pmatrix} X_1 & \cdots & O\\ & \ddots & \\ O & \cdots & X_l\end{pmatrix}$$ where $X_i$ is in the centralizer of the primary type ${\lbd^{(i)}}_{d_i}$. Therefore, the branches of such types are of the form $$(B_1 \oplus \cdots \oplus B_l)$$ where $B_i$ is a branch of ${\lbd^{(i)}}_{d_i}$, like we saw in Lemma~\ref{L1101}. Thus, with the help of Lemmas~\ref{L2cent},~\ref{L2reg},~\ref{L111}~and~\ref{L21}, we have the following results:
\begin{lemma}
 For $A$ of the type $(1,1,1)_1(1)_1$, its branching rules are given in the table below. \def\arraystretch{1.1}
 \begin{equation*}\begin{array}{cc}\hline
 \mathrm{Type} & \mathrm{Number~of~Branches}\\ \hline
(1,1,1)_1(1)_1 & q^2\\
(2,1)_1(1)_1 & q^2\\
(1,1)_1(1)_1(1)_1 & q^3-q^2\\
\Reg & q^4\\ \hline
 \end{array}
\end{equation*}

\end{lemma}

\begin{lemma}
If $A$ is of type $(2,1)_1(1)_1$, then it has $q^2$ branches of the type $(2,1)_1(1)_1$ and $q^4+q^2$ branches of the $\Reg$ type.
\end{lemma}

\begin{lemma}
If $A$ is of similarity class type $(1,1)_1(1,1)_1$, then the branching rules are given in the table below
\def\arraystretch{1.3}
\begin{equation*}\begin{array}{cc}\hline
\mathrm{Type} & \mathrm{Number~of~Branches}\\ \hline
(1,1)_1(1,1)_1 & q^2\\
(1,1)_1(2)_1 & 2q^2\\
(1,1)_1(1)_2 & q^3-q^2\\
(1,1)_1(1)_1(1)_1&q^3-q^2\\
\Reg & q^4\\ \hline
 \end{array}
\end{equation*}
\end{lemma}

\begin{lemma}\label{L11121}
If $A$ is of type $(1,1)_1(2)_1$ then it has $q^3$ branches of the type $(1,1)_1(2)_1$ and $q^4$ $\Reg$ branches.
\end{lemma}

\begin{lemma}\label{L11112}
 If $A$ is of the type $(1,1)_1(1)_2$, then it has $q^3$ branches of the type $(1,1)_1(1)_2$ and $q^4$ $\Reg$ branches.
\end{lemma}

\begin{lemma}\label{L1211}If $A$ is of similarity class type $(1,1)_1(1)_1(1)_1$, then it has $q^3$ branches of the type $(1,1)_1(1)_1(1)_1$ and $q^4$ $\Reg$ branches.\end{lemma}

\subsection{Branching Rules of the Primary Types}
We have three primary types of similarity classes in the $4\times 4$ case: $(3,1)_1$, $(2,2)_1$ and $(2,1,1)_1$. The proofs of the lemmas here will be done by the method that was used in the proof prove Lemma~\ref{L21}
\begin{lemma}
 If $A$ is of the type, $(3,1)_1$, then it has $q^3$ branches of the type $(3,1)_1$ and $q^4 + q^2$ $\Reg$ branches.
\end{lemma}
\begin{proof}
A matrix, $A$ of type $(3,1)_1$, has the canonical form:
$$\begin{pmatrix}0&1&0&0\\0&0&1&0\\0&0&0&0\\0&0&0&0\end{pmatrix}.$$
Any matrix $B \in Z(A)$ is of the form
$$B = \begin{pmatrix}a_0&a_1&a_2&b\\0&a_0&a_1&0\\0&0&a_0&0\\0&0&c&d\end{pmatrix}.$$
Let $X$ be an invertible matrix in $Z(A)$.
$$X = \begin{pmatrix}x_0&x_1&x_2&y\\0&x_0&x_1&0\\0&0&x_0&0\\0&0&z&w\end{pmatrix},\text{ where $x_0,w \neq 0$.}$$
 $$\text{Let $B ' = \begin{pmatrix}a_0'&a_1'&a_2'&b'\\0&a_0'&a_1'&0\\0&0&a_0'&0\\0&0&c'&d'\end{pmatrix}$ be the conjugate of $B$ by $X$,}$$ i.e., $XB = B'X$. Then we have the following:
$$
a_0' = a_0\text{, }d'= d \text{ and }a_1' = a_1
$$
With this, we have the following set of equations:
\begin{eqnarray}
 x_0a_2+yc&=&a'_2x_0 + b'z\label{E311}\\
x_0b+yd &=& b'w + a_0y\label{E312}\\
za_0 +wc &=& c'x_0 + dz\label{E313}
\end{eqnarray}
We will count the number of branches by looking at the following cases:
$$a_0 = d\text{ and }a_0 \neq d.$$
\textbf{When $a_0 = d$:} We get $x_0b = b'w$ and $wc = c'x_0$. So, we look at the cases, $b=c=0$ and $(b,c)\neq (0,0)$ separately.\\

 \underline{$b = c = 0$:} In this case equation~(\ref{E311}) boils down to $x_0a_2 = a'_2x_0$, therefore $a'_2 = a_2$. Therefore, any matrix in $Z(A)$ commutes with $B$. Hence, $Z(A,B) = Z(A)$. Therefore $(A,B)$ is of the type $(3,1)_1$. So, there are $q\times q \times q = q^3$ branches of this type.\\

\underline{$(b,c) \neq (0,0)$:} First we assume that $b \neq 0$. Then equation~(\ref{E312}) boils down to $x_0b = b'w$. As $b$ is non zero, choose $x_0 = w/b$ so that $b' =1$. Letting $b = b' =1$, we get $x_0 = w$. Therefore equation~(\ref{E313}) boils down to $x_0c = c'x_0$, which implies: $c' =c$. Hence equation~(\ref{E311}) boils down to $x_0a_2+yc=a'_2x_0 + z$. So, choose a $z$ such that $a_2 = 0$. Then $B$ is reduced to $$ \begin{pmatrix}a_0& a_1&0&1\\0&a_0&a_1&0\\0&0&a_0&0\\0&0&c&a_0\end{pmatrix}.$$ Hence $Z(A, B)$ is  $$\left\{\begin{pmatrix}x_0&x_1&x_2&y\\0&x_0&x_1&0\\0&0&x_0&0\\0&0&cy&x_0\end{pmatrix}\text{ : $x_0, x_1,x_2,y \in \F_q$ }\right\}.$$ It is 4 dimensional and commutative (by a routine check). Hence we have a $\Reg$ branch and there are $q\times q^2 = q^3$ such $\Reg$ branches.\\

Next, we assume that $b = 0$ and $c \neq 0$. Then equation~(\ref{E313}) boils down to $wc = c'x_0$. Like in the previous case, we can choose an approriate $w$ such that $c' = 1$. Letting $c = c' = 1$, we get $x_0 = w$. Then equation~(\ref{E311}) gets reduced to $a'_2x_0 = a_2x_0 + y$. So, we can choose $y$ such that $a_2x_0 + y =0$ . This gives us $a'_2 = 0$. Our $B$ is reduced to, $$\begin{pmatrix}a_0&a_1&0&0\\0&a_0&a_1&0\\0&0&a_0&0\\0&0&1&a_0\end{pmatrix}.$$ Its centralizer in $Z(A)$ is  $$Z(A,B) = \left\{ \begin{pmatrix}x_0&x_1&x_2&0\\0&x_0&x_1&0\\0&0&x_0&0\\0&0&z&x_0\end{pmatrix}\text{ : } x_0,x_1,x_2,z\in \F_q\right\}.$$ $Z(A,B)$ is 4 dimensional and commutative. The similarity class of $(A,B)$ is of a $\Reg$ type and there are $q^2$ such branches.\\

\textbf{When $a_0 \neq d$:} Using the fact that $a_0 - d \neq 0$, in equation~(\ref{E312}), we can choose $y$ such that $b'$ becomes 0 and in equation~(\ref{E313}), we can choose a suitable $z$ such that $c'$ becomes 0. Therefore, equation~(\ref{E311}) boils down to $x_0a_2 = a'_2x_0$. Thus giving us $a_2 =a'_2$. Hence, $B$ is reduced to $$\begin{pmatrix}a_0&a_1&a_2&0\\0&a_0&a_1&0\\0&0&a_0&0\\0&0&0&d\end{pmatrix}.$$ Its centralizer in $Z(A)$ is $$ \left\{\begin{pmatrix}x_0&x_1&x_2&0\\0&x_0&x_1&0\\0&0&x_0&0\\0&0&0&w\end{pmatrix}\text{ : } x_0,x_1,x_2,w \in \F_q\right\}.$$ This centralizer is that of a matrix of similarity class type $(3)_1(1)_1$, which is a $\Reg$ type. This gives us $q\times (q-1)\times q^2 = q^4 - q^3$ such $\Reg$ branches. So, adding up all the $\Reg$ branches, we get the total number of $\Reg$ branches $A$ to be $$(q^4 -q^3) + q^3 + q^2 = q^4 + q^2.$$
\end{proof}
\begin{lemma}\label{L221}
For $A$ of similarity class type, $(2,2)_1$, its branching rules are given in the table below.
\def\arraystretch{1.26}
\begin{equation*}\begin{array}{|cc|cc|}\hline
\mathrm{Type} & \mathrm{No.~of~Branches}&\mathrm{Type} & \mathrm{No.~of~Branches}\\ \hline
(2,2)_1 & q^2&\mathrm{New~type~NT2} & \frac{q^3-q^2}{2}\\
\Reg & q^4&\mathrm{New~type~NT3} &  \frac{q^3 - q^2}{2} \\
\mathrm{New~type~NT1} & q^2 & ~ &~\\
 \hline
 \end{array}
\end{equation*}
\begin{itemize}
\item The centralizer algebra of $\mathrm{NT1}$ is: $$\left\{\begin{pmatrix}                                                                                                                         x_0& x_1 & x_2 & x_3\\ 0&x_0&y_2&y_3\\0&0&x_0&x_1\\0&0&0&x_0\end{pmatrix} \text{ : } x_i, y_j \in \F_q\text{ for } i = 0,1,2,3 \text{ and } j =2,3\right\},$$ and its group of units is therefore of size $q^6-q^5$

\item The centralizer algebra of $\mathrm{NT2}$ is: $$\left\{\begin{pmatrix}                                                                                                                            p(C)& X\\ \rmO & p(C)\end{pmatrix} \text{ : } p(C) \in \F_q[C] \text{ and } X\in M_2(\F_q)\right\},$$ and its group of units is therefore of size $q^6-q^4$.
Here $C$ is a $2 \times 2$ matrix of the type $(1)_2$.

\item The centralizer algebra of $\mathrm{NT3}$ is $$\left\{\begin{pmatrix}x_0&0&y_1&y_2\\0&x_1&y_3&y_4\\0&0&x_0&0\\0&0&0&x_1
\end{pmatrix} \text{ : } x_i, y_j \in \F_q\text{ for } i = 0,1 \text{ and } j =1,2,3,4\right\},$$ and its group of units is therefore of size $q^4(q-1)^2$.
\end{itemize}

\end{lemma}
\begin{proof}
 A matrix $A$ of similarity class type $(2,2)_1$, is of the form $$A = \begin{pmatrix}0&1&0&0\\0&0&0&0\\0&0&0&1\\0&0&0&0\end{pmatrix}.$$ Observe that, by conjugating $A$ by an elementary matrix such that its 2nd and 3rd rows (resp. columns) are switched, gives us $$\begin{pmatrix}\rmO & I_2\\ \rmO& \rmO\end{pmatrix}, $$ where $I_2$ is the $2\times 2$ identity matrix. Thus $Z(A)$ is $$Z(A) =\left\{ \begin{pmatrix}C & D \\ \rmO & C\end{pmatrix} \text{ :$C, D \in M_2(\F_q)$ } \right\}.$$ Now, two matrices, $B = \begin{pmatrix}C&D\\ \rmO &C\end{pmatrix}$ and $B' = \begin{pmatrix}C'&D'\\ \rmO &C'\end{pmatrix} \in Z(A)$, are similar if there is an invertible matrix $\begin{pmatrix}X&Y\\ \rmO&X\end{pmatrix}$ (where $X$ is invertible), such that $$\begin{pmatrix}C'&D'\\ \rmO&C'\end{pmatrix}\begin{pmatrix}X&Y\\ \rmO&X\end{pmatrix} = \begin{pmatrix}X&Y\\\rmO&X\end{pmatrix}\begin{pmatrix}C&D\\ \rmO&C\end{pmatrix},$$ which on expanding gives us $$\begin{pmatrix}C'X & C'Y + D'X \\ \rmO & C'X\end{pmatrix} = \begin{pmatrix}XC & XD + YC \\ \rmO & XC\end{pmatrix},$$ which means that $C'$ and $C$ have to be similar. Now, we shall see the similarity classes when
\begin{enumerate}
\item $C$ is of central type.
\item $C$ is of $\Reg$ type.
\end{enumerate}
When $C$ is of central type, we have $D'X = XD$. So we only need to look at different types of $D$. Hence, to find out which matrix commutes with $\begin{pmatrix}C&D\\\rmO &C\end{pmatrix}$, we need to find the $X$ that commutes with $D$.\\

 When $D$ is of the central type, then $X$ can be any $2 \times 2$ invertible matrix in $Z(A)$. Hence the centralizer  algebra of $\begin{pmatrix}C&D\\ \rmO&C\end{pmatrix}$ in $Z(A)$, is $Z(A)$ itself, which is isomorphic to the centralizer of the type $(2,2)_1$. Therefore the branch $(A,B)$ is of type $(2,2)_1$. The number of such branches is $q \times q = q^2$.\\

When $D$ is of the type $(2)_1$ i.e., $D = \begin{pmatrix}d&1\\0&d\end{pmatrix}$, then $XD = DX$ if and only if $X = \begin{pmatrix}x_0&x_1\\0&x_0\end{pmatrix}$  Thus the centralizer algebra of $B$ in $Z(A)$ is $$ Z(A, B) = \left\{ \begin{pmatrix}x_0&x_1&y_1&y_2\\0&x_0&y_3&y_4\\0&0&x_0&x_1\\0&0&0&x_0\end{pmatrix}\text{ : $x_i, y_j \in \F_q$}\right\}.$$ Thus, the size of the centralizer group, $Z(A,B)^*$, is $(q-1)\times q^5 = q^6 - q^5$. But none of the known types of $4\times 4$ matrices have centralizer groups of size $q^6 - q^5$. We thus have a new type of similarity class of pairs of commuting matrices. This is our new type $\mathrm{NT1}$. There are $q^2$ such branches.\\

Next, if $D$ is of type, $(1)_2$, then the matrices, $X$, that commute with $D$ are polynomials in $D$, i.e., $xI_2 + yD$ where $x, y \in \F_q$. Therefore, $$Z(A,B) = \left\{\begin{pmatrix}xI_2+yD & Y \\ \rmO & xI_2 + yD\end{pmatrix}\text{ $\mid$ $x,y \in \F_q$, $Y \in M_2(\F_q)$}\right\}.$$ It can be shown that $xI + yD$ is invertible iff $(x,y) \neq (0,0)$. Thus, the centralizer group, $Z(A,B)^*$, of $B$ in $Z(A)$ has $q^4 \times (q^2 -1) = q^6 - q^4$ matrices, which is not the size of the centralizer group of any known type in $M_4(\F_q)$. Thus we have, ${q\choose2} \times q = \frac{1}{2}(q^3 - q^2)$ branches of a new type which we shall refer to as $\mathrm{NT2}$.\\
When $D$ is of type $(1)_1(1)_1$ i.e., $D = \begin{pmatrix}d_1 & 0\\0 & d_2\end{pmatrix}$, where $d_1 \neq d_2$: $X$ commutes with $D$ iff $X = \begin{pmatrix}x_0&0\\0&x_1\end{pmatrix}$. So, the common centralizer, $Z(A,B)$, of $A$ and $B$ is $$\left\{\begin{pmatrix}x_0&0&y_1&y_2\\0&x_1&y_3&y_4\\0&0&x_0&0\\ 0&0&0&x_1\end{pmatrix}\text{ : }x_i,y_j \in \F_q\right\},$$ and the size of the centralizer group, $Z(A,B)^*$, is $(q-1)^2 \times q^4$ which is the same as that of the centralizer group of $(3,1)_1$. But $Z(A,B)$ is not isomorphic to the centralizer of $(3,1)_1$, and there is no other similarity class type other than $(3,1)_1$ in $M_4(\F_q)$, whose centralizer group is of size, $q^4(q-1)^2$. Hence we have a new type which we call, $\mathrm{NT3}$. There are $q\times {q\choose 2} = \frac{1}{2}(q^3 - q^2)$ branches of this new type.\\

 Now, when $C$ is any of the $\Reg$ types of matrices: \\
$C$ is of type $(2)_1$, $C$ is of the form $\begin{pmatrix}a_0&1\\0&a_0\end{pmatrix}$ and so for $X$ to commute with $C$, we must have $X = \begin{pmatrix}x&y\\0&x\end{pmatrix}$ where $x \neq 0$. So we have the following equation: $$\begin{pmatrix}a_0&1&c'_1&c'_2\\0&a_0&c'_3&c'_4\\0&0&a_0&1\\0&0&0&a_0\end{pmatrix}\begin{pmatrix}x&y&z_1&z_2\\0&x&z_3&z_4\\0&0&x&y\\0&0&0&x\end{pmatrix} = \begin{pmatrix}x&y&z_1&z_2\\0&x&z_3&z_4\\0&0&x&y\\0&0&0&x\end{pmatrix}\begin{pmatrix}a_0&1&c_1&c_2\\0&a_0&c_3&c_4\\0&0&a_0&1\\0&0&0&a_0\end{pmatrix}.$$ Then we get $c_3 = c'_3$ and the following equations:
\begin{eqnarray}
 c'_1 x + z_3 &=& xc_1 + yc_3\label{E2211}\\
 c'_2x + c'_1y + z_4 &=& xc_2 + yc_4 + z_1\label{E2212}\\
 c'_4 x + c'_3 y &=& xc_4 + z_3\label{E2213}.
\end{eqnarray}
 Then, in equation~(\ref{E2211}) we can choose $z_3$ so that $c'_1 = 0$. Letting $c_1 = 0$, we have $z_3 = c_3y$. Then equation~(\ref{E2213}) becomes $c'_4x = xc_4$, and therefore $c'_4 = c_4$.  In equation~(\ref{E2212}) , we can choose $z_4$ such that $c'_2 = 0$. Thus, $B$ gets reduced to $$B = \begin{pmatrix}a_0&1&0&0\\0&a_0&c_3&c_4\\0&0&a_0&1\\0&0&0&a_0\end{pmatrix}.$$ Thus, $Z(A,B)$ is$$\left\{\begin{pmatrix} x & y & z_1 & z_2\\ 0& x& c_3y& c_4y + z_1\\0&0&x&y\\0&0&0&x\end{pmatrix}\text{ : } x,y,z_1,z_2 \in \F_q\right\},$$ which is 4-dimensional and commutative (a routine check). Hence the pair $(A,B)$ is of $\Reg$ type, and there are $q^3$ such branches.\\

If $C$ is of type $(1)_1(1)_1$, $C$ is of the form, $\begin{pmatrix}a_0&0\\0&c\end{pmatrix}$, where $c \neq a_0$. Any matrix that commutes with $C$ is of the form $\begin{pmatrix}x&0\\0&y\end{pmatrix}$. Now we have: $$\begin{pmatrix}a_0&0&d'_1&d'_2\\0&c&d'_3&d'_4\\0&0&a_0&0\\0&0&0&c\end{pmatrix}\begin{pmatrix}x&0&z_1&z_2\\0&y&z_3&z_4\\0&0&x&0\\0&0&0&y\end{pmatrix} = \begin{pmatrix}x&0&z_1&z_2\\0&y&z_3&z_4\\0&0&x&0\\0&0&0&y\end{pmatrix}\begin{pmatrix}a_0&0&d_1&d_2\\0&c&d_3&d_4\\0&0&a_0&0\\0&0&0&c\end{pmatrix}.$$  This gives us $d'_1 = d_1$ and $d'_4 = d_4,$ and the following equations:
\begin{eqnarray}
cz_3 + d'_3x &=& yd_3 + a_0z_3, \label{E2214}
\\ a_0z_2 + d'_2y &=& xd_2 +z_2c.\label{E2215}
\end{eqnarray}
Using the fact that $c \neq a_0$, we can get rid of $d_2$ and $d_3$ and reduce $B$ to $$B = \begin{pmatrix}a_0&0&d_1&0\\0&c&0&d_4\\0&0&a_0&0\\0&0&0&c\end{pmatrix}.$$ Thus, $Z(A,B)$ is $$\left\{\begin{pmatrix}x &0 & z_1 & 0\\0&y&0&z_4\\0&0&x&0\\0&0&0&y\end{pmatrix}\text{ : $x, y, z_1,z_4 \in\F_q$}\right\}.$$ If we conjugate any matrix in the above algebra, by the elementary matrix such that the 2nd and 3rd rows (resp. columns) are switched, then we get $$\left\{\begin{pmatrix}x&z_1&0&0\\0&x&0&0\\0&0&y&z_4\\0&0&0&y\end{pmatrix} \text{ $\mid$ $x_0,y,z_1,z_4 \in \F_q$}\right\},$$ which is the centralizer algebra of the type $(2)_1(2)_1$. Hence, this branch is of the $\Reg$ type $(2)_1(2)_1$. The number of branches of this type is $q^2 \times {q\choose2} = \frac{1}{2}(q^4 - q^3)$.\\

 When $C$ is of type, $(1)_2$, we may take $C$ to be the companion matrix of its characteristic polynomial, $f$ (a degree 2 irreducible polynomial over $\F_q$). Then from the equation below, $$\begin{pmatrix}C_f&D'\\0&C_f\end{pmatrix}\begin{pmatrix}X&Y\\0&X\end{pmatrix} = \begin{pmatrix}X&Y\\0&X\end{pmatrix}\begin{pmatrix}C_f&D\\0&C_f\end{pmatrix},$$ we have $C_fY + D'X = XD + YC_f$ (Here $X$ is a polynomial in $C_f$). We get 4 equations (by equating the 4 entries) and using the fact that the constant part of $f$ is non-zero (since it is irreducible), we can reduce $\begin{pmatrix}C_f&D\\0&C_f\end{pmatrix}$ to $$B = \begin{pmatrix}C_f  & D'\\0&C_f\end{pmatrix},$$ where $D' = \begin{pmatrix}d_1&0\\d_2&0\end{pmatrix}$. So, $Z(A,B)$ is $$\left\{\begin{pmatrix}x_0I + x_1C_f & x_1D' + y_0I + y_1C_f\\0&x_0I+ x_1C_f\end{pmatrix}\text{ : }x_0,x_1,y_0,y_1\in \F_q\right\}.$$ $Z(A,B)$ is 4-dimensional and commutative (again a routine check). Therefore we have a $\Reg$ type of branch, and there are $q^2\displaystyle{q\choose 2} = \displaystyle\frac{q^4-q^3}{2}$ such branches. Adding up the $\Reg$ branches gives us: $$\frac{q^4-q^3}{2} + \frac{q^4-q^3}{2} + q^3 = q^4 \text{ $\Reg$ branches.}$$
\end{proof}
\begin{lemma}\label{L211}
  For $A$ of similarity class type $(2,1,1)_1$, the branching rules are given in Table \ref{Tabl9}.
\begin{table}[h]
\def\arraystretch{1.2}
\begin{equation*}\begin{array}{| cc | cc |}\hline
\mathrm{Type} & \mathrm{No.~of~Branches}& \mathrm{Type} & \mathrm{No.~of~Branches}\\ \hline
(2,1,1)_1 & q^2 &\mathrm{NT1}& q\\
(3,1)_1 & q^2-q& \mathrm{NT3}& q^2\\
(1,1)_1(2)_1 & q^3-q^2& \mathrm{New~type~NT4}& q\\
(2,1)_1(1)_1 & q^3-q^2& \mathrm{New~type~NT5}& q\\
\mathrm{Regular} & q^4 + q^2 &~ & ~\\ \hline
\end{array}
\end{equation*}
\caption{Branching Rules of type $(2,1,1)_1$}
\label{Tabl9}
\end{table}
\begin{itemize}
\item The centralizer algebra of the new type, $\mathrm{NT4}$, is of the form $$\left\{\begin{pmatrix}x_0 & x_1 &x_2&x_3\\0&x_0&0&0\\0&z_1&z_2&z_3\\0&0&0&x_0\end{pmatrix} \text{ : $x_i,z_j \in \F_q$ for $i = 0,1,2,3$ and $j= 1,2,3$ }\right\}.$$
\item The centralizer algebra of the new type, $\mathrm{NT5}$, is of the form $$\left\{\begin{pmatrix}x_0&0&x_1&x_2\\0&x_0&x_3&x_4\\0&0&y_1&y_2\\0&0&0&x_0\end{pmatrix} \text{ : $x_i,y_j \in \F_q$ for $i = 0,1,2,3,4$ and $j= 1,2$ }\right\}.$$
\end{itemize}
\end{lemma}
\begin{proof} Matrix $A$ of the type $(2,1,1)_1$ has the canonical form
$$A = \begin{pmatrix}0&1&0&0\\0&0&0&0&\\0&0&0&0\\0&0&0&0\end{pmatrix},$$
Any matrix, $B\in Z(A)$, is of the form $$\begin{pmatrix}a_0&a_1&a_2&a_3\\0&a_0&0&0\\0&b_1&b_2&b_3\\0&c_1&c_2&c_3\end{pmatrix}.$$ Conjugating $B$ by an elementary matrix (denote it by $E_{234}$), such that, the 2nd row (column) moves to the 3rd row (resp. column), the 3rd row (column) moves to the 4th row (resp. column) and the 4th row (column) moves to the 2nd row (resp. column), gives us: $$B = \begin{pmatrix}a_0&\overrightarrow{b}^T&a_1\\ \overrightarrow{0}&C&\overrightarrow{d}\\0&\overrightarrow{0}^T&a_0\end{pmatrix},$$ where $\overrightarrow{b}^T = \begin{bmatrix}b_1&b_2\end{bmatrix}$, $\overrightarrow{d} = \begin{bmatrix}d_1\\d_2\end{bmatrix}$, and $C$ is a $2\times 2 $ matrix.\\
Let $$X = \begin{pmatrix}x_0&\overrightarrow{y}^T&x_1\\ \overrightarrow{0}&Z&\overrightarrow{w}\\0&\overrightarrow{0}^T&x_0\end{pmatrix}$$ be an invertible matrix in $Z(A)$. Conjugate $B$ by $X$ to get $B'$, which is $$B' = \begin{pmatrix}a_0&\overrightarrow{b'}^T&a_1\\ \overrightarrow{0}&C'&\overrightarrow{d'}\\0&\overrightarrow{0}^T&a_0\end{pmatrix}.$$ Then $XB = B'X$ gives us the following:
\begin{eqnarray}
C'Z &=& ZC\label{E2111},\\
a_0\overrightarrow{y}^T + \overrightarrow{b'}^T Z &=& x_0\overrightarrow{b}^T + \overrightarrow{y}^TC,\label{E2112}\\
C'\overrightarrow{w} + x_0\overrightarrow{d'} &=& Z\overrightarrow{d} + a_0\overrightarrow{w},\label{E2113}\\
 \overrightarrow{b'}^T.\overrightarrow{w} + a'_1x_0 &=& x_0a_1 + \overrightarrow{y}^T.\overrightarrow{d}. \label{E2114}
\end{eqnarray}
To get the branches, we will analyze the different types of $C$. To begin with, there are two main cases of $C$.
\begin{enumerate}
\item $a_0$ is an eigenvalue of $C$. Here the types of $C$ are: $C$ is $\Cent$ (ie., $C = a_0I$), $C$ is of the $\Reg$ types, $(2)_1$ and $(1)_1(1)_1$.
\item $a_0$ is not an eigenvalue of $C$. Here the types of $C$ are: $C$ is $\Cent$ (i.e, $C = cI$,$c \neq 0$), $C$ is of the $\Reg$ types, $(2)_1$, $(1)_1(1)_1$ and $(1)_2$.
\end{enumerate}
We can take $C' = C$ and therefore $Z$ is a matrix which commutes with $C$. So for each type of $C$, we will only need to see what $Z$ is and simplify $B$ only using equations~(\ref{E2112}) and~(\ref{E2113}).\\

\textbf{When $a_0$ is an eigenvalue of $C$}: We will first see the branching rules in the case where $a_0$ is an eigenvalue of $C$. So we have the following subcases:
 \begin{itemize}
 \item $\overrightarrow{b} = \overrightarrow{d} = \overrightarrow{0}$
 \item $(\overrightarrow{b}, \overrightarrow{d}) \neq (\overrightarrow{0}, \overrightarrow{0})$
 \end{itemize}
\textbf{Case: $\overrightarrow{b} = \overrightarrow{d} = \overrightarrow{0}$}. In this case, equation~(\ref{E2114}) is reduced to $$x_0a'_1 = x_0a_1.$$ This implies $a'_1 = a_1$. Now we see what happens for different types of $C$.
 \begin{description}
  \item[When $C$ is central] We have $$B = \begin{pmatrix}a_0& \overrightarrow{0}^T&a_1\\ \overrightarrow{0}& a_0I & \overrightarrow{0}\\ 0& \overrightarrow{0}^T&a_0\end{pmatrix}.$$ Here,  equations~(\ref{E2112}) and~(\ref{E2113}) are void. Thus any $X$ in $Z(A)$ commutes with $B$. $Z(A,B)$ is $Z(A)$ itself. Therefore $(A,B)$ is of type $(2,1,1)_1$ and the number of such branches is $q\times q = q^2$.
  \item[When $C$ is of type $(2)_1$] We have $$C = \begin{pmatrix}a_0&1\\0&a_0\end{pmatrix}\text{, and } B = \begin{pmatrix}a_0 & 0&0 & a_1\\ 0 & a_0 &1&0\\0&0&a_0&0\\0&0&0&a_0\end{pmatrix}.$$ So $Z = \begin{pmatrix}z_1&z_2\\0&z_1\end{pmatrix}$, and equations~(\ref{E2112}) and~(\ref{E2113}) become
   \begin{eqnarray*}
   \begin{pmatrix}a_0y_1&a_0y_2\end{pmatrix} &=& \begin{pmatrix}y_1&y_2\end{pmatrix}\begin{pmatrix}a_0&1\\0&a_0\end{pmatrix},\\ \begin{pmatrix}a_0w_1\\a_0w_2\end{pmatrix} &=& \begin{pmatrix}a_0&1\\0&a_0\end{pmatrix}\begin{pmatrix}w_1\\w_2\end{pmatrix}.
   \end{eqnarray*}
   This gives us $y_1 = w_2 = 0$. Therefore, a matrix in $Z(A,B)$ is of the form, $$\begin{pmatrix}x_0&0&y_2&x_1\\0&z_1&z_2&w_1\\0&0&z_1&0\\0&0&0&x_0\end{pmatrix}.$$  Conjugating this matrix in $Z(A,B)$ by elementary matrices (by switching the 3rd and 4th rows (resp. columns)), gives us $$\begin{pmatrix}x_0&0&x_1&y_2\\0&z_1&w_1&z_2\\0&0&x_0&0\\0&0&0&z_1\end{pmatrix},$$  which is a matrix in the common centralizer of pair of commuting matrices of the new type, $\mathrm{NT3}$. Thus the commuting pair $(A,B)$ is of similarity class type $\mathrm{NT3}$, and there are $q^2$ branches of the new type $\mathrm{NT3}$.
  \item[ When $C$ is of type $(1)_1(1)_1$], $C = \begin{pmatrix}a_0&0\\0&c\end{pmatrix}$ where, $c \neq a_0$ and $B$ becomes $$B = \begin{pmatrix}a_0 & 0&0 & a_1\\ 0 & a_0 &0&0\\0&0&c&0\\0&0&0&a_0\end{pmatrix}.$$ $Z$ commutes with $C$ iff $Z = \begin{pmatrix}z_1&0\\0&z_4\end{pmatrix}$.\\
  From equations~(\ref{E2112}) and~(\ref{E2113}), we have the following:
  \begin{eqnarray*}
  \begin{pmatrix}a_0y_1 & a_0y_2\end{pmatrix} &=& \begin{pmatrix}y_1&y_2\end{pmatrix}\begin{pmatrix}a_0&0\\0&c\end{pmatrix}\text{ and}\\
  \begin{pmatrix}a_0w_1\\a_0w_2\end{pmatrix} &=& \begin{pmatrix}a_0&0\\0&c\end{pmatrix}\begin{pmatrix}w_1\\w_2\end{pmatrix},
  \end{eqnarray*}
   which leaves us with $y_2 = w_2 = 0$ (since $a_0 \neq c$) and therefore, $Z(A,B)$ consists of $X$ of the form $$X = \begin{pmatrix}x_0&y_1&0&x_1\\0&z_1&0&w_1\\0&0&z_2&0\\0&0&0&x_0\end{pmatrix}.$$ Conjugating this by the matrix $E_{234}$, gives us $$\begin{pmatrix}x_0&x_1&y_1&0\\0&x_0&0&0\\0&w_1&z_1&0\\0&0&0&z_2\end{pmatrix},$$ which is in the centralizer algebra of a matrix of the type $(2,1)_1(1)_1$. Hence the commuting pair $(A,B)$ is of the type $(2,1)_1(1)_1$ and we have $q^3-q^2$ branches of this type.
  \end{description}
\textbf{Case: $(\overrightarrow{b},\overrightarrow{d})\neq (\overrightarrow{0}, \overrightarrow{0})$}: In this case, we can find a suitable $\overrightarrow{y}$ or $\overrightarrow{w}$ in equation~(\ref{E2114}) and get rid of the entry, $a_1$. So our $B$ is: $$\begin{pmatrix}a_0&\overrightarrow{b}&0\\ \overrightarrow{0}&C&\overrightarrow{d}\\0&\overrightarrow{0}^T&a_0\end{pmatrix}.$$
  \begin{description}
  \item[When $C = a_0I$] $Z$ is any $2\times2$ invertible matrix.
   We first assume $\overrightarrow{b} \neq \overrightarrow{0}$.\\
   Equation~(\ref{E2112}) becomes $$\overrightarrow{b'}^TZ = x_0\overrightarrow{b}^T.$$  We may replace $Z$ by $x_0^{-1}Z$ so that we have $$\overrightarrow{b'}^TZ = \overrightarrow{b}^T \text{ and } Z\overrightarrow{d} = \overrightarrow{d}'.$$ Since $\overrightarrow{b}\neq \overrightarrow{0}$ and $Z$ is invertible, we can find a suitable $Z$ such that $\overrightarrow{b'}^T = \begin{pmatrix}1&0\end{pmatrix}$. Now, let $\overrightarrow{b}^T = \overrightarrow{b'}^T = \begin{pmatrix}1&0\end{pmatrix}$, then equation~(\ref{E2112}) gives us $Z = \begin{pmatrix}1&0\\z_3&z_4\end{pmatrix}$. Hence, equation~(\ref{E2113}) boils down to
   \begin{equation}\label{E2115}
    \begin{pmatrix}1&0\\z_3&z_4\end{pmatrix}\begin{pmatrix}d_1\\d_2\end{pmatrix} = \begin{pmatrix}d'_1\\d'_2\end{pmatrix},
   \end{equation}
    $$\text{therefore }\begin{pmatrix}d'_1\\d'_2\end{pmatrix} = \begin{pmatrix}d_1\\z_3d_1 + z_4d_2\end{pmatrix}$$
   If $\overrightarrow{d} \neq \overrightarrow{0}$, with $d_1 \neq 0$, then choose $z_3$ so that $z_4d_2 + d_1z_3= 0$. Therefore, $d'_2 = 0$, and $B$ is reduced to $$\begin{pmatrix}a_0&1&0&0\\0&a_0&0&d_1\\0&0&a_0&0\\0&0&0&a_0\end{pmatrix},$$ and any $X \in Z(A,B)$ is of the form $$X = \begin{pmatrix}x_0&y_1&y_2&x_1\\0&x_0&0&d_1x_1\\0&0&z_4&w_2\\0&0&0&x_0\end{pmatrix}.$$ Conjugate this by the elementary matrix (by switching the 3rd and 4th rows (resp. columns)). Then we get: $$\begin{pmatrix}x_0&y_1&x_1&y_2\\0&x_0&d_1x_1&0\\0&0&x_0&0\\0&0&w_2&z_4\end{pmatrix},$$ which is in the centralizer of a matrix of type, $(3,1)_1$. Hence $(A,B)$ is of type $(3,1)_1$. There are $q(q-1) =q^2-q$ branches of this type.\\

 Now when $\overrightarrow{d} \neq 0$ and $d_1 =0$, equation~(\ref{E2115}) becomes $$\begin{pmatrix}d'_1\\d'_2\end{pmatrix} = \begin{pmatrix}0\\z_4d_2\end{pmatrix},$$ which can be reduced to $\begin{pmatrix}0\\1\end{pmatrix}$. Thus $B$ is reduced to $$\begin{pmatrix}a_0&1&0&0\\0&a_0&0&0\\0&0&a_0&1\\0&0&0&a_0\end{pmatrix},$$ and a matrix in $Z(A,B)$ is of the form $$X = \begin{pmatrix}x_0&y_1&y_2&x_1\\0&x_0&0&y_2\\0&z_3&x_0&w_2\\0&0&0&x_0\end{pmatrix}.$$ On conjugating $X$ by the elementary matrices such that its 2nd and 3rd rows and columns are switched, we get $$\begin{pmatrix}x_0&y_2&y_1&x_1\\0&x_0&z_3&w_2\\0&0&x_0&y_2\\0&0&0&x_0\end{pmatrix},$$ which is in the centralizer of a pair of commuting matrices of the new type $\mathrm{NT1}$. Hence $(A,B)$ is of type $\mathrm{NT1}$, and we have $q$ branches of the new type $\mathrm{NT1}$.\\

 When $\overrightarrow{d} = \overrightarrow{0}$, then $$B = \begin{pmatrix}a_0&1&0&0\\0&a_0&0&0\\0&0&a_0&0\\0&0&0&a_0\end{pmatrix}$$ whose centralizer, $Z(A,B)$ in  $Z(A)$ contains matrices of the form, $$X = \begin{pmatrix}x_0&y_1&y_2&x_1\\0&x_0&0&0\\0&z_3&z_4&w_2\\0&0&0&x_0\end{pmatrix}.$$ So the common centralizer algebra of $A$ and $B$ is 7~dimensional. As there is no known type in $M_4(\F_q)$ whose centralizer is 7 dimensional, we have a new type, which we call $\mathrm{NT4}$. There are $q$ branches of this type.\\

   Next, when $\overrightarrow{b} = \overrightarrow{0}$: Here $\overrightarrow{d} \neq \overrightarrow{0},$ and from equation~(\ref{E2113}), we can find $Z$ such that $Z\overrightarrow{d} = \begin{pmatrix}1\\0\end{pmatrix},$ and our $B$ is reduced to $$ \begin{pmatrix}a_0&0&0&0\\0&a_0&0&1\\0&0&a_0&0\\0&0&0&a_0\end{pmatrix}.$$ Thus $Z(A,B)$ has matrices of the form $$\begin{pmatrix}x_0&0&y_2&x_1\\0&x_0&z_2&w_1\\0&0&z_4&w_2\\0&0&0&x_0\end{pmatrix}.$$ Hence, the centralizer algebra of $(A,B)$ is 7 dimensional, but it is not conjugate to the centralizer of $\mathrm{NT4}$ and therefore, the branch is of a new type, which we shall call $\mathrm{NT5}$. There are $q$ such branches.

   \item[When $C$ is of type $(2)_1$ i.e., $C = \left(\begin{smallmatrix}a_0&1\\0&a_0\end{smallmatrix}\right)$] We have $Z = \begin{pmatrix}z_1&z_2\\0&z_1\end{pmatrix}$ where $z_1 \neq 0$. From equation~(\ref{E2112}), we get:
   \begin{equation}\label{E2116}\overrightarrow{b'}^T\begin{pmatrix}z_1&z_2\\0&z_1\end{pmatrix} + \overrightarrow{y}^T(a_0I - C) = x_0\overrightarrow{b}^T\end{equation} As $a_0I - C = \begin{pmatrix}0&-1\\0&0\end{pmatrix}$, the LHS in equation~(\ref{E2116}) above boils down to,
   \begin{equation*}
   \begin{pmatrix}b'_1z_1 & b'_1z_2 + b'_2z_1 - y_1\end{pmatrix}.
    \end{equation*}
    Choose $y_1$ so that $\overrightarrow{b}^T = \begin{pmatrix}b_1'z_1 & 0\end{pmatrix}$. We now have two cases: $b'_1 \neq 0$ and $b'_1 = 0$.\\
    When $b'_1 \neq 0$, we can choose $z_1$ so that $\overrightarrow{b}^T = \begin{pmatrix}1&0\end{pmatrix}$. Letting $\overrightarrow{b'}^T = \overrightarrow{b}^T  = \begin{pmatrix}1&0\end{pmatrix}$, equation~(\ref{E2112}) becomes $$\begin{pmatrix}z_1&z_2-y_1\end{pmatrix} = \begin{pmatrix}1&0\end{pmatrix},$$ which implies: $z_1 = 1$ and $y_1 = z_2$. So $\overrightarrow{y}^T = \begin{pmatrix}z_2 & y_2\end{pmatrix}$. \\
 Then equation~(\ref{E2113}) is reduced to $$\begin{pmatrix}1&z_2\\0&1\end{pmatrix}\begin{pmatrix}d_1\\d_2\end{pmatrix} + \begin{pmatrix}0&-1\\0&0\end{pmatrix}\begin{pmatrix} w_1\\w_2\end{pmatrix} = \begin{pmatrix}d'_1\\d'_2\end{pmatrix},$$ which implies that we can choose $w_2$ appropriately so that $\begin{pmatrix}d'_1\\d'_2\end{pmatrix} = \begin{pmatrix}0\\d_2\end{pmatrix}$. Thus $B$ is reduced to $$\begin{pmatrix}a_0&1&0&0\\0&a_0&1&0\\0&0&a_0&d_2\\0&0&0&a_0\end{pmatrix}.$$ Thus, $Z(A,B)$ is $$\left\{\begin{pmatrix}x_0&y_1&y_2&x_1\\0&x_0&y_1&d_2y_2\\0&0&x_0&d_2y_1\\0&0&0&x_0\end{pmatrix} \text{ : $x_0,x_1,y_1,y_2 \in \F_q $}\right\},$$ and it is conjugate to the centralizer of a $\Reg$ nilpotent $(4)_1$ type of matrix. This branch $(A,B)$ is of a $\Reg$ type, and there are $q\times q = q^2$ such branches.\\

  Now if $b'_1 =0$, then we have $\overrightarrow{b}^T = \overrightarrow{0}^T$. Then equation~(\ref{E2113}) becomes $$\begin{pmatrix} z_1&z_2\\0&z_2\end{pmatrix}\begin{pmatrix}d_1\\d_2\end{pmatrix} + \begin{pmatrix}0&-1\\0&0\end{pmatrix}\begin{pmatrix}w_1\\w_2\end{pmatrix} = \begin{pmatrix}d'_1&d'_2\end{pmatrix}$$ which gives us $$\overrightarrow{d'} = \begin{pmatrix}z_1d_1 + z_2d_2 - w_2\\ z_1d_2\end{pmatrix}$$ choose $w_2$ such that $\overrightarrow{d'} = \begin{pmatrix}0\\z_1d_2\end{pmatrix}$.\\
    If $d_2 \neq 0$, we can scale it to $1$ and thus we have $$B = \begin{pmatrix}a_0&0&0&0\\0&a_0&1&0\\0&0&a_0&1\\0&0&0&a_0\end{pmatrix}$$ so in this case $Z(A,b)$ is:  $$\left\{ \begin{pmatrix}x_0&0&0&x_1\\0&x_0&z_2&w_1\\0&0&x_0&z_2\\0&0&0&x_0\end{pmatrix}\text{ : $x_0,x_1,w_1,z_2 \in F_q$}\right\}.$$ It is 4-dimensional and commutative. Therefore, this branch too is of a $\Reg$ type and the number of branches is $q$. So we have a total of $q^2 + q$ branches of this $\Reg$ type.\\
    If $d_2 = 0$, we are back to the case $\overrightarrow{b} = \overrightarrow{d} = \overrightarrow{0}$.

    \item [When $C = \left(\begin{matrix}a_0&0\\0&c\end{matrix}\right)$ ($c \neq a_0$)], $Z = \begin{pmatrix}z_1&0\\0&z_4\end{pmatrix}$. So equation~(\ref{E2112}) becomes $$\begin{pmatrix}b'_1&b'_2\end{pmatrix}\begin{pmatrix}z_1&0\\0&z_4\end{pmatrix} + \begin{pmatrix}y_1&y_2\end{pmatrix}\begin{pmatrix}0&0\\0& a_0-c\end{pmatrix} = \begin{pmatrix}b_1&b_2\end{pmatrix}.$$ We get from this $$\begin{pmatrix}z_1b'_1&z_2b'_2 + (a_0 -c)y_2\end{pmatrix}  = \begin{pmatrix}b_1&b_2\end{pmatrix}.$$ As $a_0-c \neq 0$, we can get rid of $b'_2$ so that $\overrightarrow{b}^T = \begin{pmatrix}z_1b'_1&0\end{pmatrix}$. \\
    If $b'_1 \neq 0$, then we can reduce $\overrightarrow{b}^T$ to  $\begin{pmatrix}1&0\end{pmatrix}$. In equation~(\ref{E2112}), letting $\overrightarrow{b'}^T = \overrightarrow{b}^T = \begin{pmatrix}1&0\end{pmatrix}$, we get $\begin{pmatrix}z_1&(a_0 -c)y_2\end{pmatrix} = \begin{pmatrix}1&0\end{pmatrix}$. Thus $z_1 = 1$ and $y_2 =0$. So $Z = \begin{pmatrix}1&0\\0&z_4\end{pmatrix}$.\\
    Equation~(\ref{E2113}) becomes $$\begin{pmatrix}d'_1\\d'_2\end{pmatrix} = \begin{pmatrix}1&0\\0&z_4\end{pmatrix}\begin{pmatrix}d_1\\d_2\end{pmatrix} + \begin{pmatrix}0&0\\0&a_0 -c\end{pmatrix}\begin{pmatrix}w_1\\w_2\end{pmatrix}.$$ Using $a_0 \neq c$, we can reduce $\overrightarrow{d'}$ to $\begin{pmatrix}d_1\\0\end{pmatrix}$. Thus $$B = \begin{pmatrix}a_0&1&0&0\\0&a_0&0&d_1\\0&0&c&0\\0&0&0&a_0\end{pmatrix}.$$ Then $Z(A,B)$ is $$\left\{ \begin{pmatrix}x_0&y_1&0&x_1\\0&x_0&0&d_1y_1\\0&0&z_4&0\\0&0&0&x_0\end{pmatrix}\text{ $x_0,x_1,y_1,z_4 \in \F_q$} \right\}$$ Conjugating by the elementary matrices such that its 3rd and 4th rows and columns are switched, we get: $$\left\{\begin{pmatrix}x_0&y_1&x_1&0\\0&x_0&d_1y_1&0\\0&0&x_0&0\\0&0&0&z_4 \end{pmatrix}\text{ $x_0,x_1,y_1,z_4 \in \F_q$} \right\},$$ which is the centralizer of the $\Reg$ type $(3)_1(1)_1$. Therefore this branch is of $\Reg$ type. The number of such branches is $q^2(q-1) = q^3 - q^2$.\\

     When $b'_1 = 0$, then $\overrightarrow{b}^T = \overrightarrow{0}^T$.  Then equation~(\ref{E2113}) becomes
    \begin{equation*}
    \begin{aligned}
    \begin{pmatrix}d'_1\\d'_2\end{pmatrix} &= \begin{pmatrix}z_1&0\\0&z_4\end{pmatrix}\begin{pmatrix}d_1\\d_2\end{pmatrix} + \begin{pmatrix}0&0\\0&a_0 -c\end{pmatrix}\begin{pmatrix}w_1\\w_2\end{pmatrix}\\
    &= \begin{pmatrix}z_1d_1 \\ z_4d_2 + (a_0 -c)w_2\end{pmatrix}.
    \end{aligned}
    \end{equation*}

    As $a_0 \neq c$, we can make $z_4d_2$ vanish by choosing $w_2$ appropriately. So, $\overrightarrow{d'} =\begin{pmatrix}z_1d_1\\0\end{pmatrix}$. If $d_1 \neq 0$. Choose $z_1$ so that $\overrightarrow{d'} = \begin{pmatrix}1\\0\end{pmatrix}$. Thus, $B$ is reduced to $$\begin{pmatrix}a_0&0&0&0\\0&a_0&0&1\\0&0&c&0\\0&0&0&a_0\end{pmatrix}.$$ So, here $Z(A,B)$ is $$\left\{ \begin{pmatrix}x_0&0&0&x_1\\0&x_0&0&w_1\\0&0&z_4&0\\0&0&0&x_0\end{pmatrix} \text{ $x_0,x_1,w_1,z_4 \in \F_q$} \right\},$$ which is 4 dimensional and commutative. Thus the pair, $(A,B)$, is of $\Reg$ type and there are $q(q-1) = q^2 -q$ such branches. So we have a total of $$(q^3-q^2) + (q^2 -q ) +(q^2+q)=q^3 +q^2$$ branches of the $\Reg$ type so far.\end{description}

\textbf{When $a_0$ is not an eigenvalue of $C$}: Here, $C -a_0I$ is an invertible matrix. In equations~(\ref{E2112}) and~(\ref{E2113}), using the fact that $C-a_0I$ is invertible, we can reduce $\overrightarrow{b}$ and $\overrightarrow{d}$ to $\overrightarrow{0}$. After this, equations~(\ref{E2112}) and~(\ref{E2113}) become.
     \begin{eqnarray*}
     \begin{pmatrix}y_1&y_2\end{pmatrix}(C-a_0I) &=& \begin{pmatrix}0&0\end{pmatrix}\\
     (C-a_0I)\begin{pmatrix}w_1\\w_2\end{pmatrix} &=& \begin{pmatrix}0\\0\end{pmatrix}.
     \end{eqnarray*}
Therefore $\overrightarrow{y} = \overrightarrow{w} = \overrightarrow{0}$, and equation~(\ref{E2114}) becomes $a'_1x_0 = x_0a_1$, therefore $a'_1 = a_1$. So $B$ is of the form $$\begin{pmatrix}a_0&\overrightarrow{0}^T&a_1 \\ \overrightarrow{0}&C&\overrightarrow{0}\\0 & \overrightarrow{0}^T & a_0\end{pmatrix}.$$ So the centralizers of such $B$ in $Z(A)$ are of the form $$Z(A,B) = \left\{\begin{pmatrix}x_0 &\overrightarrow{0}^T&x_1\\ \overrightarrow{0}&Z&\overrightarrow{0}\\0 &\overrightarrow{0}^T&x_0\end{pmatrix}\text{ $\mid$ $x_0, x_1 \in \F_q$, $ZC =CZ$}\right\}.$$  We can conjugate this by elementary matrices to get  $$\left\{\begin{pmatrix}x_0&x_1 &\overrightarrow{0}^T\\0 &x_0&\overrightarrow{0}^T\\ \overrightarrow{0}&\overrightarrow{0}&Z\end{pmatrix} \text{ $\mid$ $x_0, x_1 \in \F_q$, $ZC =CZ$}\right\}.$$

When $C$ is of the $\Cent$ type i.e., $C = cI$ where $c \neq a_0$, we have $Z$ to be any $2\times 2$ invertible matrix and thus $Z(A,B)$ is the centralizer of a matrix of type $(1,1)_1(2)_1$. Therefore we have a branch of type $(1,1)_1(2)_1$ and, we have $q^2(q-1)$ such branches.\\

When $C$ is of the $\Reg$ type whose eigenvalue is not $a_0$, the centralizer, $Z(A,B)$, of $B$ in $Z(A)$, consists of matrices of the form $\begin{pmatrix} Y&0\\0&p(C)\end{pmatrix}$ where $p(C)$ is a polynomial in $C$. This common centralizer of $A$ and $B$ is that of the type $(2)_1\tau$ where $\tau$ is one of $(2)_1$, $(1)_1(1)_1$ and $(1)_2$, which are $\Reg$ $2 \times 2 $ types, for which $a_0$ is not an eigenvalue. So, $(A,B)$ is of the $\Reg$ type and we therefore have, $q^2 \times (q^2 - 1-(q-1))=q^4-q^3$, such $\Reg$ branches. Adding up the number of all the $\Reg$ branches gives a total of, $$(q^4-q^3) + (q^3+q^2)=q^4 + q^2\text{ $\Reg$ branches},$$ and hence we get Table \ref{Tabl9}.
\end{proof}

\begin{lemma}\label{L1102}
 If $A$ is of type $(1,1)_2$, then it has $q^2$ branches of the type $(1,1)_2$ and $q^4$ $\Reg$ branches.
\end{lemma}
\begin{proof}
 The proof is like that of the $(1,1)_1$ case for $2\times2$ matrices over $\F_{q^2}$.
\end{proof}
\subsection{Branching Rules of the New types}\label{S44new}
While finding out the branching rules for the types, $(2,1,1)_1$ and $(2,2)_1$, we got 5 new types of branches: $\mathrm{NT1}$, $\mathrm{NT2}$, $\mathrm{NT3}$, $\mathrm{NT4}$ and $\mathrm{NT5}$. In this subsection, we will see the branching rules of those new types.
\begin{lemma}\label{LNT1} For a pair  $(A, B)$ of similarity class type $\mathrm{NT1}$, the branching rules are given in the table below:
\def\arraystretch{1.2}
\begin{equation*}\begin{array}{cc}\hline
\mathrm{Type} & \mathrm{Number~of~Branches}\\ \hline
  \mathrm{NT1} & q^3\\
\Reg & q^4-q^3\\
\mathrm{New~ Type~NT6} & q^4-q^2\\\hline
 \end{array}
\end{equation*}
The centralizer of the new type $\mathrm{NT6}$ is $$\left\{\begin{pmatrix}a_0I&C\\0&a_0I\end{pmatrix}\text{ : $a_0 \in \F_q$, $C \in M_2(\F_q)$}\right\}.$$
\end{lemma}
\begin{proof}
In this case, $$Z(A, B) = \left\{\begin{pmatrix}a_0 + a_1D&C\\0&a_0I + a_1D\end{pmatrix}~:~C \in M_2(\F_q),~a_0,~a_1 \in F_q \right\},$$ where $D = \begin{pmatrix}0&1\\0&0\end{pmatrix}$.

To see the branching rules here, we will use a different approach from what we have been using so far. Let $M= \begin{pmatrix}a_0 + a_1D&C\\0&a_0I + a_1D\end{pmatrix}$ be an invertible matrix and $X = \begin{pmatrix}x_0 + x_1D&Y\\0&x_0I + x_1D\end{pmatrix}$. We have
$$MX = \begin{pmatrix}(a_0I+a_1D)(x_0I+x_1D)& (a_0I + a_1D)Y + C(x_0+x_1D)\\0&(a_0I+a_1D)(x_0I+x_1D) \end{pmatrix},$$
$$ XM = \begin{pmatrix}(x_0I+x_1D)(a_0I+a_1D)& (x_0I + x_1D)C + Y(a_0+a_1D)\\0&(x_0I+x_1D)(a_0I+a_1D)\end{pmatrix}.$$

So, $XM = MX$ if and only if $$a_1DY +x_1CD = x_1DC + a_1YD,$$ which implies
\begin{equation}\label{ENT1}[a_1Y - x_1C, D] = 0.\end{equation}
Thus we need to deal with 4 cases of what $x_1$ and $Y$ are, in equation~(\ref{ENT1}).\\

\textbf{When $x_1 = 0$ and $[Y,D] = 0$:} There are $qq^2 = q^3$ matrices $X$ in this case and equation~(\ref{ENT1}) holds for any $a_1$ and any $C$. Thus the centralizer group, $Z(A,B,X)^*$, of $X$ in $Z(A,B)^*$ is $Z(A,B)^*$ itself.\\
Thus, under conjugation by $Z(A,B)^*$:
\begin{itemize}
\item Orbit size of $X$ = 1.
\item Number of orbits is $\displaystyle\frac{q^3}{1} = q^3$.
\end{itemize}
Thus $(A,B,X)$ is of type $\mathrm{NT1}$ and the number of branches is $q^3$\\

\textbf{When $x_1 = 0$ and $[Y,D]\neq 0$:} The number of $X$'s is $q(q^4 -q^2)$. Thus, equation~(\ref{ENT1}) boils down to $a_1[Y,D] = 0$. But $[Y,D] \neq 0$ implies $a_1 = 0$. Thus $Z(A,B,X)^*$ is $$\left\{\begin{pmatrix}a_0I&C\\0&a_0I\end{pmatrix} \text{ : $a_0 \neq 0$} \right\}.$$ The size of $Z(A,B,X)^*$ is, $(q-1)q^4 = q^5 - q^4$. But none of the types of similarity classes (the known types and the 5 new types), has a 5-dimensional centralizer algebra . So we now have another new type, $\mathrm{NT6}$.
Therefore:
\begin{itemize}
\item Orbit size of $X$ = $\displaystyle\frac{q^6-q^5}{q^5-q^4} = q$.
\item Number of orbits is $\displaystyle\frac{q(q^4-q^2)}{q} = q^4 - q^2$.
\end{itemize}
Thus $(A,B,X)$ is of type $\mathrm{NT6}$ and the number of branches is $q^4 - q^2$.\\

\textbf{When $x_1 \neq 0 $ and $[Y,D] = 0$:} The number of $X$'s is $q(q-1)q^2 = q^4 - q^3$. Thus, equation~(\ref{ENT1}) boils down to $x_1[C,D] = 0$ , which means that $[C,D] = 0$. Thus, $C = b_0I + b_1D$. So $$Z(A,B,X)^* = \left\{ \begin{pmatrix}a_0I + a_1D& b_0I + b_1D\\0&a_0I +a_1D\end{pmatrix}\text{ : $a_0 \neq 0$ } \right\},$$ and this centralizer group is a commutative group of size $q^4-q^3$. So $Z(A,B,X)$ is 4-dimensional and commutative. Therefore:
\begin{itemize}
\item Orbit size of $X$ = $\displaystyle\frac{q^6-q^5}{q^4-q^3} = q^2$.
\item Number of orbits is $\displaystyle\frac{q^4-q^3}{q^2} = q^2 - q$.
\end{itemize}
Thus $(A,B,X)$ is of a $\Reg$ type, and the number of branches is $q^2 - q$.\\

\textbf{When $x_1 \neq 0 $ and $[Y,D] \neq 0$:} The number of $X's$ of this kind is $q(q-1)(q^4-q^2)$. In this case, equation~(\ref{ENT1}) remains as it is, i.e., $[a_1Y - x_1C, D] = 0$. This implies that $x_1C - a_1Y \in \F_q[D]$. $x_1\neq 0$ implies $C = x_1^{-1}a_1Y + b_0I + b_1D$. So, $$Z(A,B,X)^* = \left\{ \begin{pmatrix}a_0I + a_1D & x_1^{-1}a_1Y + b_0I + b_1D\\0&a_0I + a_1D\end{pmatrix}\text{ : $a_0\neq 0$}\right\}.$$ It is of size $q^4 - q^3$, and is commutative (a routine check). So $Z(A,B,X)$ is 4-dimensional and commutative. Therefore:\begin{itemize}
\item Orbit size of $X$ = $\displaystyle\frac{q^6-q^5}{q^4-q^3} = q^2$.
\item Number of orbits is $\displaystyle\frac{q(q-1)(q^4-q^2)}{q^2} = (q-1)(q^3-q)$.
\end{itemize}
Thus $(A,B,X)$ is of a $\Reg$ type and the number of branches is $q^2 - q$.\\
Adding up the number of branches of all the $\Reg$ types, we get a total of $(q^2 - q) + (q-1)(q^3 -q)$, which is equal to $$(q-1)(q + q^3 -q) = q^4 - q^3\text{ $\Reg$ branches.}$$  Hence we have the table mentioned in the statement.
\end{proof}

\begin{lemma}\label{LNT2} For $(A,B)$ of similarity class type $\mathrm{NT2}$, the branching rules are given in the table below
\def\arraystretch{1.2}
\begin{equation*}\begin{array}{cc}\hline
\mathrm{Type} & \mathrm{Number~of~ranches}\\ \hline
  \mathrm{NT2} & q^3\\
\Reg & q^4-q^3\\
\mathrm{NT6} & q^4-q^3\\\hline
 \end{array}
\end{equation*}
\end{lemma}
\begin{proof}
 $Z(A,B)$ is equal to $$\left\{\begin{pmatrix}a_0I + a_1C_f & D \\0& a_0I + a_1C_f\end{pmatrix}\text{ : $a_0,a_1 \in \F_q$, $D \in M_2(\F_q)$}\right\},$$ where $C_f$ is a $2\times 2$ matrix, whose characteristic polynomial is a degree 2 irreducible polynomial $f$.  A matrix in $Z(A,B)$ is invertible iff $(a_0, a_1) \neq (0,0)$ and hence the size of the $Z(A,B)^*$ is $q^6 - q^4$.  To prove this lemma, we will follow the steps we used in the proof of \hyperref[LNT1]{Lemma \ref*{LNT1}}.
Let $M =\begin{pmatrix}a_0I + a_1C_f & D \\0& a_0I + a_1C_f\end{pmatrix}$ be invertible and let $X =\begin{pmatrix}x_0I + x_1C_f & Y \\0& x_0I + x_1C_f\end{pmatrix}$. Then $M$ and $X$ commute iff \begin{equation}\label{ENT2} [a_1Y - x_1D, C_f] = 0.\end{equation} From equation~(\ref{ENT2}), we have 4 cases for what $x_1$ and $Y$ should be:
We shall analyze the cases:\\

\textbf{When $x_1 = 0$ and $[Y,C_f] =0$:} The number of $X$'s is $qq^2 = q^3$. Here, equation~(\ref{ENT2}) holds for any $a_1$ and any $D$. Thus the centralizer group, $Z(A,B,X)^*$, of $X$ in $Z(A,B)^*$ is the whole of $Z(A,B)^*$. Thus there are $q^3$ orbits under the conjugation by $Z(A,B)^*$.Therefore the triple $(A,B,X)$ is of type $\mathrm{NT2}$. Hence we have $q^3$ branches of type $\mathrm{NT2}$.\\

\textbf{When $x_1 = 0$ and $[Y,C_f] \neq 0$:} The number of matrices $X$ is $q(q^4 - q^2)$. Equation ~(\ref{ENT2})~ boils down to $a_1[Y,C_f] = 0$ which implies $a_1 = 0$. Thus $$Z(A,B,X)^*= \left\{\begin{pmatrix}a_0I&B\\0&a_0I\end{pmatrix}\text{: $a_0 \in \F_q$,$B \in M_2(\F_q)$ }\right\}$$ and its size is $(q-1)q^4 = q^5 - q^4$. So $(A,B,X)$ is of class type $\mathrm{NT6}$. From this we get:
\begin{itemize}
\item Orbit size of $X$ = $\displaystyle\frac{q^6-q^4}{q^5 -q^4} = q+1$.
\item Number of such orbits = $\displaystyle\frac{q(q^4-q^2)}{q+1} = q^4-q^3$.
\end{itemize}
The number of branches of type $\mathrm{NT6}$ is $q^4-q^3$.\\

\textbf{When $x_1 \neq 0$ and $[Y,C_f] = 0$:} The number of matrices $X$ is $$q(q-1)q^2 = q^4-q^3.$$ From equation~(\ref{ENT2}), $x_1[D, C_f] = 0$ , which implies $[D,C_f] = 0$. Hence $D = d_0I + d_1 C_f$ and therefore $$Z(A,B,X)^* = \left\{ \begin{pmatrix}a_0I + a_1C_f & d_0I + d_1C_F\\0&a_0I +a_1C_f\end{pmatrix}\text{ : $(a_0,a_1)\neq (0,0)$}\right\}.$$ It is commutative and its size is $(q^2-1)q^2 = q^4 - q^2$. Hence, $Z(A,B,X)$ is of dimension 4. So, the triple $(A,B,X)$ is a branch of a $\Reg$ type. The size of the orbit of $X$ is $\displaystyle\frac{q^6-q^4}{q^4 - q^2} = q^2$. There are $$\frac{q^4 - q^3}{q^2} = q^2 -q$$ branches of this $\Reg$ type.\\

\textbf{When $x_1\neq 0 $ and $[Y,C_f] \neq 0$:} The number of matrices is $$q(q-1)(q^4 - q^2).$$ Equation~(\ref{ENT2}) gives us, $D \in x_1^{-1}a_1Y + \F_q[C_f]$. So, $Z(A,B,X)^*$ is:  $$\left\{\begin{pmatrix}a_0I + a_1C_f& x_1^{-1}a_1Y + d_0I + d_1C_F\\0& a_0I + a_1C_f\end{pmatrix} \text{ : $(a_0,a_1) \neq (0,0)$}\right\}.$$ It is commutative and its size is $(q^2-1)q^2 = q^4 - q^2$. Thus, the algebra $Z(A,B,X)$ is of dimension 4. Thus, this branch too is $\Reg$. The size of the orbit of $X$ in $Z(A,B)$ is $(q^6-q^4)/(q^4 - q^2) = q^2$ and the number of orbits is therefore $q(q-1)(q^4 - q^2)/q^2 = q(q-1)(q^2 -1)$.
Therefore, the total number  of $\Reg$ branches is $$q(q-1)(q^2 -1) + (q^2 -q) = q^4 - q^3$$
Thus we have the table mentioned in the statement.
\end{proof}
\begin{lemma}\label{LNT3}
 If $A$ is of similarity class type $\mathrm{NT3}$, then its branching rules are given in the table below:
\def\arraystretch{1.2}
\begin{equation*}\begin{array}{cc}\hline
\mathrm{Type} & \mathrm{Number~of~Branches}\\ \hline
  \mathrm{NT3} & q^3\\
\Reg & q^4-q^3\\
\mathrm{New~Type~NT6} & q^4+q^3\\\hline
 \end{array}
\end{equation*}

\end{lemma}
\begin{proof}
 The centralizer algebra of a pair, $(A,B)$, of the new type $\mathrm{NT3}$ is $$Z(A,B) = \left\{\begin{pmatrix}D(c_0,c_1)&C\\0&D(c_0,c_1)\end{pmatrix}\text{ $\mid$ $c_0, c_1 \in \F_q$, $C\in M_2(\F_q)$}\right\},$$ where $D(c_0, c_1)$ is a $2\times 2$ diagonal matrix with $c_0$ and $c_1$ as its diagonal entries. This $D(c_0, c_1)$ can also be written as $c_0I + c_1D(0,1)$ (replace $c_1 - c_0$ by $c_1$). Let $X$ be: $$X = \begin{pmatrix}x_0I +x_1D(0,1)&Y\\0&x_0I +x_1D(0,1)\end{pmatrix}$$ and $M$ be an invertible matrix in $Z(A,B)$: $$M = \begin{pmatrix}c_0I + c_1D(0,1) & C\\0&c_0I + c_1D(0,1)\end{pmatrix}.$$ As $M$ is invertible, $c_0 \neq 0 $ and $c_0 + c_1 \neq 0$. So, $XM = MX$ iff $[c_1Y - x_1D, D(0,1)] = 0$. From this equation, we have four cases as to what $x_1$ and $Y$ have to be, i.e.,\\

 \textbf{ When $x_1 = 0$ and $[Y, D(0,1)] = 0$:} The number of such $X$'s is $q^3$. Here $c_1$ can be anything and $C$ can be any $2\times 2$ matrix. So the centralizer group, $Z(A,B,X)^*$ of $X$ in $Z(A,B)^*$ is $Z(A,B)^*$ itself. Therefore the orbit of $X$ is of size 1 and there are $q \times q^2  = q^3$ such orbits. Hence $q^3$ branches of type $\mathrm{NT3}$.\\

\textbf{ When $x_1 = 0$ and $[Y, D(0,1)] \neq 0$:} The number of such $X$'s is $q(q^4-q^2)$. $c_1[Y,D(0,1)] = 0$ implies $c_1 = 0$. Thus, $$Z(A,B,X)^* = \left\{ \begin{pmatrix}c_0I&C\\0&c_0I\end{pmatrix}: \text{ $c_0 \neq 0$} \right\}.$$  Thus $(A,B,X)$ is of the type $\mathrm{NT6}$. Its orbit size is $\displaystyle\frac{q^4(q-1)^2}{q^4(q-1)} = q-1$ and there are $q\times(q^4-q^2)$ such matrices. Hence the number of orbits is $\displaystyle\frac{q^3(q^2-1)}{q-1} = q^3(q+1) =q^4 +q^3$. We therefore have $q^4 +q^3$ branches of this new type.\\

\textbf{ When $x_1 \neq 0$ and $[Y, D(0,1)] = 0$:} There are $q(q-1)q^2$ such matrices and we have $x_1[D, D(0,1)] = 0$ which implies that $C = d_0I + d_1D(0,1)$. Hence, $$Z(A,B,X)^* = \left\{\begin{pmatrix} c_0I +c_1D(0,1) & d_0I +d_1D(0,1)\\0&c_0 I + c_1D(0,1)\end{pmatrix}\text{ : $c_0 \neq 0$, $c_1 \neq -c_0$}\right\}.$$ Its size is $q^2(q-1)^2$. It is of dimension 4 and it is commutative. Therefore, $(A,B,X)$ is a $\Reg$ branch of $(A,B)$. Each orbit is of size $\displaystyle\frac{q^4(q-1)^2}{q^2(q-1)^2} = q^2$ and therefore the number of branches is $$q(q-1)q^2 /q^2 = q^2 -q.$$

\textbf{When $x_1 \neq 0$ and $[Y,D(0,1)] \neq 0$:} There are $q(q-1)(q^4-q^2)$ such $X$ and $D \in x_1^{-1}c_1Y + \F_q[D(0,1)]$. Thus $C = x_1^{-1}c_1Y + d_0I + d_1D(0,1)$ and so the $Z(A,B,X)^*$ consists of matrices of the form $$\begin{pmatrix}c_0I + c_1D(0,1) & x_1^{-1}c_1Y + d_0I + d_1D(0,1)\\0 & c_0I + c_1D(0,1)\end{pmatrix}.$$ Its size is $q^2(q-1)^2$, it is of dimension 4 and it is commutative. Thus $(A,B,X)$ is a $\Reg$ branch. The size of its orbit is $q^2$ and there are a total of $\displaystyle\frac{q(q-1)(q^4-q^2)}{q^2} = q(q-1)(q^2-1)$. On adding up all the $\Reg$ branches, we have a total of $$q(q-1)(q^2) + q(q-1) = q^4 -q^3\text{ $\Reg$ branches.}$$
\end{proof}

\begin{lemma}
For the commuting pair $(A,B)$ of similarity class type $\mathrm{NT4}$ or $\mathrm{NT5}$, there are:
\begin{itemize}
\item $q^3$ branches of its own type.
\item $q^3 + q^2$ branches of the new type $\mathrm{NT6}$.
\item $q^4$ branches of the $\Reg$ type.
\end{itemize}\end{lemma}
\begin{proof}
The proof is the same for both $\mathrm{NT4}$ and $\mathrm{NT5}$. So it will suffice to prove for any one of them. We shall prove it for $\mathrm{NT4}$.\\
 $Z(A,B)$ consists of matrices of the form $$M  = \begin{pmatrix}a_0&b_1&b_2&b_3\\0&a_0&0&0\\0&c_1&c_2&c_3\\0&0&0&a_0\end{pmatrix},$$ which on conjugation by elementary matrices (which switches the 2nd and 3rd rows and columns of $M$) becomes $$M = \begin{pmatrix}a_0&b_2&b_1&b_3\\0&c_2&c_1&c_3\\0&0&a_0&0\\0&0&0&a_0\end{pmatrix}.$$ We shall rewrite $M$ as $$\begin{pmatrix}a_0&a_1&b_1&b_2\\0&b_0&b_3&b_4\\0&0&a_0&0\\0&0&0&a_0\end{pmatrix},$$ and let $M'$ be a conjugate of $M$ in $Z(A,B)$: $$M' = \begin{pmatrix}a'_0&a'_1&b'_1&b'_2\\0&b'_0&b'_3&b'_4\\0&0&a'_0&0\\0&0&0&a'_0\end{pmatrix}.$$ Then there is an invertible $X$ such that $XM = M'X$. Let $$X = \begin{pmatrix}x_0&x_1&y_1&y_2\\0&y_0&y_3&y_4\\0&0&x_0&0\\0&0&0&x_0\end{pmatrix},$$ where $x_0, y_0 \neq 0$. Expanding $XM = M'X$ gives us $a'_0 = a_0$ and $b'_0 = b_0$ and the following equations:
\begin{eqnarray}
 a_0x_1 + a'_1y_0 &=& a_1x_0 + x_1b_0 \label{ENT41}\\
 a'_1y_3 + b'_1x_0 &=& x_0b_1 + x_1b_3 \label{ENT42}\\
 a'_1y_4 + b'_2x_0 &=& x_1b_4 + b_2x_0 \label{ENT43}\\
b_0y_3 + b'_3x_0 &=& y_0b_3 + y_3a_0 \label{ENT44}\\
 b_0y_4 + b'_4x_0 &=& y_0b_4 + y_4a_0 \label{ENT45}
\end{eqnarray}
We have two main cases: $a_0 \neq b_0$ and $a_0 = b_0$.\\

\noindent If $a_0 \neq b_0$. Then, in equation~(\ref{ENT41}), using a suitable choice of $x_1$, we can make $a'_1 = 0$. With a suitable choice of $y_3$ in equation~(\ref{ENT44}), we can make $b'_3 = 0$. Similarly, in equation~(\ref{ENT45}), choose a suitable $y_4$ so that $b'_4 = 0$. Then from equations~(\ref{ENT42}) and~(\ref{ENT43}), we get $b'_1= b_1$ and $b'_2=b_2$. So $$M = \begin{pmatrix}a_0&0&b_1&b_2\\0&b_0&0&0\\0&0&a_0&0\\0&0&0&a_0\end{pmatrix}.$$ Its centralizer in $Z(A,B)$ is $$Z(A,B,M)=\left\{\begin{pmatrix}x_0&0&y_1&y_2\\0&y_0&0&0\\0&0&x_0&0\\0&0&0&x_0\end{pmatrix} \text{ : $x_0, y_0,y_1,y_2 \in \F_q$}\right\},$$ which is 4-dimensional and commutative. Therefore, this branch $(A,B,M)$ is of a $\Reg$ type and there are $q^3(q-1) = q^4 -q^3$ such branches\\

\noindent If $a_0 = b_0$. Then equation~(\ref{ENT41}) becomes $a'_1y_0 = a_1x_0$. Here again, there are two cases:
$$ a_1 \neq 0 \text{, and } a_1 = 0.$$
When $a_1 \neq 0$, choose $y_0$ such that $a'_1 = 1$. So, letting  $a_1 = a'_1 = 1$, we have $y_0 = x_0$. Then, from equations~(\ref{ENT44}) and~(\ref{ENT45}) we get $b'_3= b_3$ and  $b'_4 = b_4$. Equation~(\ref{ENT42}) becomes $y_3 + b'_1x_0 = x_0b_1 + x_1b_3$ and equation~(\ref{ENT43}) becomes $y_4 + b'_2x_0 = x_1b_4 + b_2x_0$. So we can choose $y_3$ and $y_4$ appropriately so that $b'_1 =b'_2 = 0$ So our $M$ reduces to $$\begin{pmatrix}a_0&1&0&0\\0&a_0&b_3&b_4\\0&0&a_0&0\\0&0&0&a_0\end{pmatrix}.$$ Hence, $$Z(A,B,M) = \left\{\begin{pmatrix}x_0&x_1&y_1&y_2\\0&x_0&x_1b_3&x_1b_4\\0&0&x_0&0\\0&0&0&x_0\end{pmatrix} \text{ : $x_0,x_1,y_1,y_2 \in\F_q $}\right\},$$ which is 4 dimensional and commutative. Thus the branch, $(A,B,M)$, is of $\Reg$ type. The number of such branches is $q^3$. So we have a total of $q^4 -q^3 + q^3 = q^4$ $\Reg$ branches.\\

\noindent When $a_1 = 0$, equation~(\ref{ENT42}) becomes $b'_1x_0 = x_0b_1 + x_1b_3$, equation~(\ref{ENT43}) becomes $b'_2x_0 = x_0b_2 + x_1b_4$ and we have from the equations~(\ref{ENT44})~and~(\ref{ENT45}), $b'_3x_0 = y_0b_3$ and $b'_4x_0 = y_0b_4$. So we can divide this into two cases.$$(b_3,b_4) = (0,0) \text{ and } (b_3,b_4) \neq (0,0)$$
When $(b_3,b_4) = (0,0)$ we have $b'_1 = b_1$ and $b'_2 = b_2$ and thus $M$ reduces to $$\begin{pmatrix}a_0&0&b_1&b_2\\0&a_0&0&0\\0&0&a_0&0\\0&0&0&a_0\end{pmatrix}.$$ Thus, $Z(A,B,M)$ is the whole of $Z(A,B)$. Thus $(A,B,M)$ is of the type $\mathrm{NT4}$ and we have $q^3$ such branches.\\

\noindent When $(b_3,b_4) \neq (0,0)$ and $b_3 \neq 0$. Then, in equation~(\ref{ENT44}), using a suitable $y_0$, we can make $b'_3 = 1$. Letting $b'_3 = b_3 = 1$, we get $y_0 = x_0$ and therefore $b'_4 = b_4$. Equation~(\ref{ENT42}) becomes $b'_1x_0 = x_0b_1 + x_1$, hence we can get $b'_1 = 0$. Letting $b_1 = b'_1 = 0$, we get $x_1 = 0$, and therefore we get $b'_2 =b_2$ (from equation (\ref{ENT43})). Thus $M$ is reduced to $$\begin{pmatrix}a_0&0&0&b_2\\0&a_0&1&b_4\\0&0&a_0&0\\0&0&0&a_0\end{pmatrix},$$ and $Z(A,B,M)$ is $$\left\{\begin{pmatrix}x_0I&Y\\ \rmO&x_0I\end{pmatrix} \text{ : $x_0,\in \F_q$, $Y \in M_2(\F_q)$}\right\},$$ which is that of the new type $\mathrm{NT6}$. Therefore this branch, $(A,B,M)$ is of type $\mathrm{NT6}$ and we have $q^3$ such branches.\\

\noindent If $b_3 = 0$ and $b_4 \neq 0$. Then we can make $b_4 = 1$ and by the arguments like in the above case, we can make $b_2 = 0$ and $b'_1 = b_1$. So $$M = \begin{pmatrix}a_0&0&b_1&0\\0&a_0&0&1\\0&0&a_0&0\\0&0&0&a_0\end{pmatrix}.$$ So, $Z(A,B,M)$ is $$\left\{\begin{pmatrix}x_0I&Y\\ \rmO&x_0I\end{pmatrix} \text{ : $x_0,\in \F_q$, $Y \in M_2(\F_q)$}\right\}.$$ Thus this $(A,B,M)$ too is a branch of the new type $\mathrm{NT6}$ and there are $q^2$ such branches. So in total we have $q^3 +q^2$ branches of the new type $\mathrm{NT6}$.
\end{proof}

\begin{lemma}
 For a triple $(A,B,M)$ of similarity class type $\mathrm{NT6}$, there are $q^5$ branches of the type $\mathrm{NT6}$.
\end{lemma}
\begin{proof}
 We know that $$Z(A,B,M) = \left\{ \begin{pmatrix}a_0I &C\\0&a_0 I\end{pmatrix}\text{ : $a_0 \in \F_q$ and $C \in M_2(\F_q)$} \right\}.$$ It is easy to see that this algebra is commutative. Hence, there is only one branch and it is of the type $\mathrm{NT6}$ and there are $q^5$ of them.
\end{proof}
We therefore have no more new similarity class types.
\subsection{Calculating $c_{4,k}(q)$}
Now, that we have all the branching rules, we can form a matrix, $\Cal{B}_4= [b_{ij}]$, with rows and columns indexed by the types. For a given type $j$, $b_{ij}$ is the number of similarity class type $i$ branches of a tuple of similarity class type $j$. This $\Cal{B}_4$ is our branching matrix. Table \ref{Tabl4} lists the rcfs, and under each rcf, it has a list of the types with that rcf. Let each of the new types be treated as separate rcf's. By the averaging technique discussed in Section \ref{S33}, we can reduce $\Cal{B}_4$ to a 11$\times$11 matrix indexed by the 5 rcfs and the 6 new types in the order:
$$\{(1,1,1,1), (2,1,1), (2,2), (3,1), (4), \mathrm{NT1},\mathrm{NT2},\mathrm{NT3},\mathrm{NT4},\mathrm{NT5},\mathrm{NT6}\}.$$

\textbf{rcf $(1,1,1,1)$:} For rcf $(1,1,1,1)$, there is only one type, which is the $\Cent$ type, $(1,1,1,1)_1$. It has $q$ branches of rcf $(1,1,1,1)$; $q^2$ branches each of rcf types, $(2,1,1)$ and $(2,2)$; $q^3$ branches with rcf $(3,1)$, and $q^4$ branches with rcf, $(4)$ (the $\Reg$ type of branches).\\

\textbf{rcf $(4)$:} The $\Reg$ type of similarity class is of rcf-type, $(4)$. It has $q^4$ branches of rcf $(4)$.\\

\textbf{rcf $(2,1,1)$:} An element of rcf type $(2,1,1)$ is of class type $(1,1,1)_1(1)_1$ with probability $\displaystyle\frac{q-1}{q}$ and of class type $(2,1,1)_1$ with probability $\displaystyle\frac{1}{q}$. So, on an average, a tuple of rcf type $(2,1,1)$ has:
\begin{itemize}
\item $q^2$ branches of rcf type $(2,1,1)$.
\item $q^3+q^2-q-1$ branches of rcf type $(3,1)$.
\item $q^4+q$ $\Reg$ (rcf type $(4)$) branches.
\item 1 branch each of types $\mathrm{NT1}$, $\mathrm{NT4}$ and $\mathrm{NT5}$.
\item $q$ branches of type $\mathrm{NT3}$.
\end{itemize}

\textbf{rcf $(2,2)$:} There are three similarity class types with rcf, $(2,2)$. They are $(1,1)_1,(1,1)_1$, $(2,2)_1$ and $(1,1)_2$. An element of rcf type, $(2,2)$, is of class type, $(1,1)_1(1,1)_1$, with probability, $\displaystyle\frac{(q-1)}{2q}$; of class type, $(2,2)_1$, with probability, $\displaystyle\frac{1}{q}$, and is of class type, $(1,1)_2$, with probability, $\displaystyle\frac{q-1}{2q}$.
So on an average, a tuple of rcf-type $(2,2)$ has:
\begin{itemize}
\item $q^2$ branches of rcf type $(2,2)$.
\item $q^3 -q^2$ branches of rcf $(3,1)$.
\item $q^4$ $\Reg$ branches.
\item $q$ branches of the new type $\mathrm{NT1}$
\item $(q^2 -q)/2$ branches each of the new types $\mathrm{NT2}$ and $\mathrm{NT3}$.
\end{itemize}
\ \\

\textbf{rcf $(3,1)$:} The similarity class types with rcf $(3,1)$ are: \begin{itemize}\item $(3,1)_1$\item $(2,1)_1(1)_1$\item $(1,1)_1(2)_1$\item $(1,1)_1(1)_2$ and \item$(1,1)_1(1)_1(1)_1$\end{itemize} Their probabilities are mentioned in the table below.
\def\arraystretch{1.2}
\begin{equation*}
\begin{array}{lc}\hline
\text{Class Type}& \text{Probability}\\ \hline
(3,1)_1 & \frac{1}{q^2}\\
(2,1)_1(1)_1 & \frac{q-1}{q^2}\\
(1,1)_1(2)_1 & \frac{q-1}{q^2}\\
(1,1)_1(1)_2 & \frac{q-1}{2q}\\
(1,1)_1(1)_1(1)_1 &\frac{(q-1)(q-2)}{2q^2}\\[0.25em] \hline
\end{array}\end{equation*}
All these types have branches of their own respective types and $\Reg$ branches. Hence we have on an average:  $q^3$ branches of rcf type $(3,1)$ and $q^4 + q$ branches of rcf type $(4)$.\\

So our branching matrix $\Cal{B}_4$ is equal to $$\left(\begin{smallmatrix} q & 0 & 0 & 0 & 0 & 0 & 0 & 0 & 0 & 0 & 0 \\q^2 & q^2 & 0 & 0 & 0 & 0 & 0 & 0 & 0 & 0 & 0 \\ q^2 & 0 & q^2  & 0 & 0 & 0 & 0 & 0 & 0 & 0 & 0 \\ q^3&q^3+q^2-q-1&q^3-q^2&q^3&0&0&0&0&0&0&0\\q^4&q^4+q&q^4&q^4+q&q^4&q^4-q^3&q^4-q^3&q^4-q^3&q^4&q^4&0\\0&1&q&0&0&q^3&0&0&0&0&0\\0&0&\frac{q^2-q}{2}&0&0&0&q^3&0&0&0&0\\0&q&\frac{q^2-q}{2}&0&0&0&0&q^3&0&0&0\\0&1&0&0&0&0&0&0&q^3&0&0\\0&1&0&0&0&0&0&0&0&q^3&0\\0&0&0&0&0&q^4-q^2&q^4-q^3&q^4+q^3&q^3+q^2&q^3+q^2&q^5\end{smallmatrix}\right).$$
Let $e_1$ denote the $11\times1$ column matrix with first entry being 1 and the rest, 0. Let $\mathbf{1}'$ denote the $1\times 11$ row matrix, whose entries are all 1's. Then we have $$c_{4,k}(q) = \mathbf{1}'\Cal{B}_4^ke_1.$$ The table below lists $c_{4,k}(q)$ for $k = 1,2,3,4$. The calculations were done using sage.
\def\arraystretch{1.32}
\begin{equation*}
 \begin{array}{lc}
  \hline
k & c_{4,k}(q)\\ \hline
  1& q^{4} + q^{3} + 2 q^{2} + q\\[0.3em]
2& q^{8} + q^{7} + 3 q^{6} + 3 q^{5} + 5 q^{4} + 3 q^{3} + 3 q^{2}\\[0.3em]
3& q^{12} + q^{11} + 3 q^{10} + 4 q^{9} + 8 q^{8} + 8 q^{7} + 11 q^{6} + 8 q^{5} + 5 q^{4} + 2 q^{3}\\[0.5em]
4& q^{16} + q^{15} + 3 q^{14} + 5 q^{13} + 9 q^{12} + 12 q^{11} + 16 q^{10}   \\&+17 q^{9} + 17 q^{8} + 13 q^{7} + 9 q^{6} + 4 q^{5} + 2 q^{4}\\ \hline
 \end{array}
\end{equation*}
 We can see that $c_{4,k}(q)$ is a polynomial in $q$ with non-negative integer coefficients for $k =1,2,3,4$. But, we can't say the same about $c_{4,k}(q)$ for general $k$. So, we will have to use the generating function for $c_{4,k}(q)$, which is $$h_4(q,t) = \sum_{k=0}^\infty c_{4,k}(q)t^k = \mathbf{1}'(I-t\Cal{B}_4)^{-1}e_1.$$ In the next subsection, we will carefully examine the expression, $h_4(q,t)$.

\subsection{Non-negativity of coefficients of $c_{4,k}(q)$}
Now it remains to check if the coefficients of $h_4(q,t)$ are non-negative. The rational generating function $h_4(q,t)$ is:
$$h_4(q,t) = \frac{r_+(q,t)-r_-(q,t)}{(1-qt)(1-q^2t)(1-q^3t)(1-q^4t)(1-q^5t)},$$
where $r_+(q,t) = 1+q^2t+2q^2t^2+q^3t^2+2q^4t^2+q^6t^3$, and\\
 $r_-(q,t) = q^5t+q^7t^2+q^3t^3+2q^7t^3+2q^9t^3+q^{10}t^4.$ We have \begin{equation*}\frac{1}{(1-qt)(1-q^2t)(1-q^3t)(1-q^4t)(1-q^5t)} = \left( \sum_{k=0}^\infty\left(\sum_{j = k}^{5k}p_{5,k}(j) q^jt^k\right)\right),\end{equation*} where $p_{5,k}(j)$ denotes the number of partitions of $j$ with $k$ parts, with the maximum part being $\leq 5$. With this, \begin{equation*}h_4(q,t) = (r_+(q,t)- r_-(q,t)) \left[1 +\left( \sum_{k=1}^\infty\left(\sum_{j = k}^{5k}p_{5,k}(j) q^jt^k\right)\right)\right].\end{equation*} Expanding this gives us
\begin{equation}\label{PC}
\begin{matrix}
 \left( \sum_{k=0}^\infty\left(\sum_{j = k}^{5k}p_{5,k}(j) q^jt^k\right)\right) & - \left( \sum_{k=0}^\infty\left(\sum_{j = k}^{5k}p_{5,k}(j) q^{j+5}t^{k+1}\right)\right)\\
+\left( \sum_{k=0}^\infty\left(\sum_{j = k}^{5k}p_{5,k}(j) q^{j+2}t^{k+1}\right)\right) &-\left( \sum_{k=0}^\infty\left(\sum_{j = k}^{5k}p_{5,k}(j) q^{j+7}t^{k+2}\right)\right) \\
 +\left( \sum_{k=0}^\infty\left(\sum_{j = k}^{5k}2p_{5,k}(j) q^{j+2}t^{k+2}\right)\right) &-\left( \sum_{k=0}^\infty\left(\sum_{j = k}^{5k}p_{5,k}(j) q^{j+3}t^{k+3}\right)\right)\\
 +\left( \sum_{k=0}^\infty\left(\sum_{j = k}^{5k}p_{5,k}(j) q^{j+3}t^{k+2}\right)\right) & -\left( \sum_{k=0}^\infty\left(\sum_{j = k}^{5k}2p_{5,k}(j) q^{j+7}t^{k+3}\right)\right)\\
 +\left( \sum_{k=0}^\infty\left(\sum_{j = k}^{5k}p_{5,k}(j) q^{j+4}t^{k+2}\right)\right) & -\left( \sum_{k=0}^\infty\left(\sum_{j = k}^{5k}2p_{5,k}(j) q^{j+9}t^{k+3}\right)\right)\\
+\left( \sum_{k=0}^\infty\left(\sum_{j = k}^{5k}p_{5,k}(j) q^{j+6}t^{k+3}\right)\right) & -\left( \sum_{k=0}^\infty\left(\sum_{j = k}^{5k}p_{5,k}(j) q^{j+10}t^{k+4}\right)\right).
\end{matrix}
\end{equation}
The coefficient, $d_{jk}$, of $q^jt^k$ in equation~(\ref{PC}) is \begin{equation}\label{PC2}\begin{aligned}d_{jk} =  (p_{5,k}(j) - p_{5,k-1}(j-5)) &+ (p_{5,k-1}(j-2) -p_{5,k-2}(j-7))\\ + ( 2p_{5,k-2}(j-2) - p_{5,k-3}(j-3))&+(p_{5,k-2}(j-3)- 2p_{5,k-3}(j-7)) \\+(2p_{5,k-2}(j-4) -2p_{5,k-3}(j-9) ) &+ (p_{5,k-3}(j-6) -  p_{5,k-4}(j-10)).\end{aligned}\end{equation}
Here are some observations which will be enough to prove that equation~(\ref{PC2}) is non-negative.
\begin{lemma}\label{Ljk}
 For any $k \geq 1$, any $j:$ $k \leq j \leq 5k$, and any $l$ such that, $1\leq l \leq 5$, $p_{5,k}(j) \geq p_{5,k-1}(j-l)$.
\end{lemma}
\begin{proof}
 We assume that $j-l \leq 5(k-1)$ so that $p_{5,k-1}(j-l) \neq 0$. Given a partition of $j-l$ with $k-1$ parts with maximal part $\geq 5$, we can attach the part $l$ to this partition to get a partition of $j$ in $k$ parts, with maximal part $\leq 5$. Hence $p_{5,k}(j) \geq p_{5,k-1}(j-l)$.
\end{proof}
As a consequence of the above lemma, we have the following inequalities.
\begin{eqnarray}
 p_{5,k}(j) &\geq& p_{5,k-1}(j-5)\label{PC51}\\
p_{5,k-1}(j-2) &\geq& p_{5,k-2}(j-7)\label{PC52}\\
p_{5,k-2}(j-2) &\geq& p_{5,k-3}(j-3)\label{PC53}\\
p_{5,k-2}(j-3) &\geq& p_{5,k-3}(j-7)\label{PC54}\\
p_{5,k-2}(j-4) &\geq& p_{5,k-3}(j-9)\label{PC55}\\
p_{5,k-3}(j-6) &\geq& p_{5,k-4}(j-10)\label{PC56}.
\end{eqnarray}
\begin{lemma}\label{LMain}
 Let $k \geq 4$. Then for $j$ such that $j-7 \geq k-3$ we have the following:
\begin{itemize}
 \item If $j-7 = 5(k-3)$, then \begin{equation}\label{PC57}(p_{5,k}(j) - p_{5,k-1}(j-5)) + (p_{5,k-2}(j-3) - 2p_{5,k-3}(j-7)) \geq 0\end{equation}
\item If $j-7 < 5(k-3)$ then \begin{equation}\label{PC58}
                              p_{5,k-2}(j-3) - 2p_{5,k-3}(j-7) \geq 0
                             \end{equation}
\end{itemize}
\end{lemma}
\begin{proof} When $j-7 = 5(k-3)$, given the only partition of $j-7$ with $k-3$ parts, we can attach two 1's to it, to get a partition of $j-5$ in $k-1$ parts. Hence $p_{5,k-1}(j-5) \geq p_{5,k-3}(j-7)$.
\begin{equation*}
\begin{aligned}
\text{Therefore } (p_{5,k}(j) - p_{5,k-1}(j-5)) &+ (p_{5,k-2}(j-3)- 2p_{5,k-3}(j-7))\\
\geq p_{5,k}(j) - 2p_{5,k-1}(j-5) &+ (p_{5,k-2}(j-3)-p_{5,k-3}(j-7)).
\end{aligned}
\end{equation*}
 Obsereve: $j-7 = 5k-15~\Rightarrow~j-5 = 5k -13 = 5(k-1)-8$. So any partition of $j-5$ with $k-1$ parts, with maximal part $\leq$ 5, will have atleast two parts, $<$ 5. So, to each of these, we can either attach a 5, or add 1 each to the two parts which are less than 5 and attach 3 as the $k$th part. This gives 2 partitions of $j$ having $k$ parts. So, $p_{5,k}(j) - 2p_{5,k-1}(j-5) \geq 0$ and therefore  \begin{equation*}
 \begin{aligned}
 (p_{5,k}(j) - p_{5,k-1}(j-5)) &+ (p_{5,k-2}(j-3)- 2p_{5,k-3}(j-7))\\
  \geq p_{5,k}(j) - 2p_{5,k-1}(j-5) &+ (p_{5,k-2}(j-3)-p_{5,k-3}(j-7)) \\
  \geq 0\text{ Since }(p_{5,k-2}(j-3)&-p_{5,k-3}(j-7)) \geq 0\text{~~(from ineq.~(\ref{PC54})). }
  \end{aligned}
  \end{equation*}
  Hence inequality~(\ref{PC57}) holds.\\

 When $j-7 < 5(k-3)$, then, for any partition of $j-7$ with $k-3$ parts with each part being atmost 5, we have atleast one part which is strictly less than 5. Given any such partition, we can either, add 1 to the part that's $<$ 5 and attach a 3, or just attach a 4 to the existing partition, to get a partition of $j-3$ in $k-2$ parts. Hence we get two partitions of $j-3$ in $k-2$ parts. Therefore inequality~(\ref{PC58}) holds.
\end{proof}

Using Lemma~\ref{LMain} and inequalities~(\ref{PC51}) to~(\ref{PC56}) , we can show that the coefficient of $q^jt^k$ for each $j, k \geq 0$, is non-negative. So for each $k\geq 1$, the coefficients of $c_{4,k}(q)$ are the coefficients of $q^jt^k$ as $j$ varies, which are non-negative. Therefore, the coefficients of $c_{4,k}(q)$ are non-negative integers.\\
Thus, Theorem~\ref{Tmain} is proved for $n = 4$.

\subsection*{Acknowledgements} I thank my supervisor, Prof.~Amritanshu Prasad for several helpful discussions, help with sage programming and for feedback on the draft of the paper. I thank the editors and the referee for going through the draft of this paper, and their feedback with suggestions for revision. I thank the Institute of Mathematical Sciences, Chennai for the hospitality and providing an excellent research environment.
\bibliographystyle{abbrv}
\bibliography{References}

\end{document}